\documentclass[11pt,reqno]{amsart}
\usepackage{amssymb, hyperref}
\def\norm#1{\|#1\|}
\def\normo#1{\left\|#1\right\|}

\def\set#1{\left\{#1\right\}}
\def\bra#1{\langle#1\rangle}
\def\wt#1{\widetilde{#1}}
\def\wh#1{\widehat{#1}}

\newcommand{\R}{{\mathbb R}}
\newcommand{\T}{{\mathbb T}}
\newcommand{\Z}{{\mathbb Z}}
\newcommand{\N}{{\mathbb N}}
\newcommand{\ft}{{\mathcal{F}}}

\newcommand{\px}{\partial_x}
\newcommand{\pt}{\partial_t}

\newcommand{\la}{\lambda}
\newcommand{\La}{{\Lambda}}

\numberwithin{equation}{section}
\newtheorem{theorem}{Theorem}[section]
\newtheorem{proposition}[theorem]{Proposition}
\newtheorem{lemma}[theorem]{Lemma}

\newtheoremstyle{definition}{}{}%
     {}
     {}
     {\bfseries}
     {. }
     {0em}
     {}

\theoremstyle{definition}
\newtheorem{definition}[theorem]{Definition}
\newtheorem{remark}[theorem]{Remark}

\def\MB{\mathbb}
\def\MC{\mathcal}

\def\SK{S_{H}}

\def\SN{S^{N}_{H}}
\def\HS{H^{-\frac{1}{2}}_0}

\begin{document}
\title[Global well-posedness and Nonsqueezing]
  {Global well-posedness and Nonsqueezing property for the higher-order KdV-type flow}

\author[S. Hong]{Sunghyun Hong}
\email{shhong7523@kaist.ac.kr}
\address{Department of Mathematical Sciences, Korea Advanced Institute of Science and Technology, 291 Daehak-ro, Yuseong-gu, Daejeon 305-701, Republic of Korea}

\author[C. Kwak]{Chulkwang Kwak}
\email{ckkwak@kaist.ac.kr}
\address{Department of Mathematical Sciences, Korea Advanced Institute of Science and Technology, 291 Daehak-ro, Yuseong-gu, Daejeon 305-701, Republic of Korea}

\begin{abstract}
In this paper, we prove that the periodic higher-order KdV-type equation 
\[\left\{\begin{array}{ll}
\pt u + (-1)^{j+1} \px^{2j+1}u + \frac12 \px(u^2)=0, \hspace{1em} &(t,x) \in \R \times \T, \\
u(0,x) = u_0(x), &u_0 \in H^s(\T),
\end{array}
\right.\]
is globally well-posed in $H^s$ for $s \ge -\frac{j}{2}$, $j \ge 3$. The proof is based on "I-method" introduced by Colliander et al. \cite{CKSTT1}.  We also prove the nonsqueezing property of the periodic higher-order KdV-type equation. The proof relies on Gromov's nonsqueezing theorem for the finite dimensional Hamiltonian system and an approximation argument for the solution flow. More precisely, after taking the frequency truncation to the solution flow, we apply the nonsqueezing theorem. By using the approximation argument, we extend this result to the infinite dimensional system. This argument was introduced by Kuksin \cite{Kuksin:1995ue} and made concretely by Bourgain \cite{Bourgain:1994tr} for the 1D cubic NLS flow, and Colliander et al. \cite{CKSTT3} for the KdV flow. One of our observations is that the higher-order KdV-type equation has the better modulation effect from the non-resonant interaction than that the  KdV equation has. Hence, unlike the work of Colliander et al. \cite{CKSTT3}, we can get the nonsqueezing property for the solution flow without the Miura transform.
\end{abstract}

\thanks{}
\thanks{} 

\maketitle

\tableofcontents
\section{Introduction}
We consider the higher-order KdV-type equation,
\begin{equation}\label{eq:kdv}
\left\{\begin{array}{ll}
\pt u + (-1)^{j+1} \px^{2j+1}u + \frac12 \px(u^2)=0, \hspace{1em} &(t,x) \in \R \times \T, \\
u(0,x) = u_0(x), &u_0 \in H^s(\T),
\end{array}
\right.
\end{equation}
for $j \in \MB{N}$ and $u$ is a real-valued function. Especially, \eqref{eq:kdv} is called KdV and Kawahara equation when $j=1,2$, respectively. These types of equations have conservation laws such as  
\begin{align}
M[u]&= \int_{\MB{T}} u dx,   &\text{(Mean)} \label{eq:conservation laws} \\
E[u] &= \int_{\MB{T}} u^2 dx,  &\nonumber\\
H[u] &= \int_{\MB{T}} \frac{1}{2} \left(\partial_x^j u\right) ^2 - \frac{1}{6} u^3 dx. &\text{(Hamiltonian)} \label{eq:Hamiltonian}
\end{align}
Furthermore, \eqref{eq:kdv} is the Hamiltonian equation with respect to \eqref{eq:Hamiltonian}. In other words, we can rewrite \eqref{eq:kdv} as follows:
\begin{equation*}
u_t = \partial_x \nabla_u H\left(u\left(t\right)\right) = \nabla_{\omega} H\left(u\left(t\right)\right)
\end{equation*}
where $\nabla_u$ is the $L^2$ gradient and $\nabla_{\omega} = \nabla_{\omega_{-\frac{1}{2}}}$ is the symplectic gradient (see \eqref{eq:symplectic form}). These three conservation laws play various roles (in particular, the global behavior) in the study on the partial differential equations. In this paper, we focus on the global well-posedness and the nonsqueezing property of \eqref{eq:kdv} for any $j \le 3$ and $j \le 2$, respectively. Thus, they are importantly used to prove our results as well.\\

\subsection{Global well-posedness}
The local and global well-posedness of \eqref{eq:kdv} were widely studied. For the local well-posedness result, Gorsky and Himonas \cite{Gorsky:2009eg} firstly proved this problem for $s\ge-\frac12$ and Hirayama \cite{Hirayama} improved for $s \ge -\frac{j}{2}$.
Both works are based on the standard Fourier restriction norm method. Hirayama improved bilinear estimate by using the factorization of the resonant function.

The results of the global well-posedness for \eqref{eq:kdv}, when $j=1,2$ were proved by Colliander et al. \cite{CKSTT1} and Kato \cite{Kato}, respectively, via "I-method". In this paper, we extend results of \cite{CKSTT1} and \cite{Kato} for $j\ge3$. The method also basically follows the argument in \cite{CKSTT1} for periodic KdV equation, while some estimates are slightly different. We encountered difficulties in the algebraic factorization of resonant functions. In order to overcome this issue, we use another argument (see Lemma \ref{lem:algebra} below) comparing with Hirayama's proof (Lemma 2.2 in \cite{Hirayama}). Remark that $s =-\frac{j}{2}$ is sharp in the sense that the bilinear estimate in $X^{s,\frac12}$ space fails for $s < -\frac{j}{2}$ (see Theorem 1.4. in \cite{Hirayama}).

The following theorem is one of the main results in this paper:
\begin{theorem}
Let $j \ge 3$ and $s \ge -\frac{j}{2}$. Then \eqref{eq:kdv} is globally well-posed in $H^s(\T)$
\end{theorem}

\subsection{Nonsqueezing property}\label{subsec:Nonsq prop}
The first contributor of the nonsqueezing property is Gromov \cite{Gromov:1985ww}. He proved the finite dimensional nonsqueezing theorem by using \emph{Darboux width}. Thereafter, Hofer and Zehnder \cite{Hofer:2011vo} developed this to the symplectic capacity. Furthermore, Kuksin \cite{Kuksin:1995ue} introduced an abstract method that the solution map of a given Hamiltonian PDE can be regarded as an approximate symplectic map on the appropriate function space. Concrete examples are presented by Bourgain \cite{Bourgain:1994tr} for the 1D cubic NLS and Colliander et al. \cite{CKSTT3} for the KdV equation. Recently, Roum\'egoux \cite{Roumegoux:2010sn} also proved the nonsqueezing property of the BBM equation and Mendelson \cite{Mendelson:2014vh} proved the nonsqueezing of the Klein-Gordon equation on $\MB{T}^3$ via a probabilistic approach. Also, the first author and Kwon \cite{HK2015} obtained the result of nonsqueezing property for the coupled KdV-type system without the Miura transform in the symplectic phase space $H^{-\frac12}(\T) \times H^{-\frac12}(\T)$.

First of all, we introduce the finite dimensional nonsqueezing theorem.

\begin{theorem}[Nonsqueezing property: finite dimensional version \cite{Kuksin:1995ue}]\label{thm:finite dimensional nonsq.}
Let $\MC{S}$ be a symplectic map on the $2n$-dimensional phase space. Let $B_R$ and $C_{k,r}$ be a ball of radius $R$ and a cylinder of radius $r$ at $k$-the component, respectively. If
\begin{equation*}
\MC{S}\left(B_R\right) \subseteq C_{k,r},
\end{equation*}
then $r \ge R$.
\end{theorem}

Intuitively, Theorem \ref{thm:finite dimensional nonsq.} means that the symplectic map cannot transform any $R$-ball into a hole of $r$-pipe placed in the \emph{basis} direction. To apply Theorem \ref{thm:finite dimensional nonsq.} to \eqref{eq:kdv}, we need a global solution in the phase space with a symplectic form and a symplectic transform. Moreover, we need appropriate truncation of the solution map on the finite dimensional function space.
We firstly find a symplectic form with respect to the given Hamiltonian \eqref{eq:Hamiltonian}. Let $\omega_{-\frac{1}{2}}$ be the symplectic form in $H_0^{-\frac{1}{2}}$ of the form of 
\begin{equation}\label{eq:symplectic form}
\omega_{-\frac{1}{2}}\left(u,v\right) := \int_{\MB{T}} u \partial_x^{-1} v dx,
\end{equation}
for all $u,v \in H_0^{-\frac{1}{2}}$.
Hence, we can rewrite \eqref{eq:kdv} as follows:
\begin{equation*}
u_t = \nabla_{\omega_{-\frac{1}{2}}} H\left(u\left(t\right)\right)
\end{equation*}
by the following observation
\begin{align*}
\omega_{-\frac{1}{2}}\left(v, \nabla_{\omega_{-\frac{1}{2}}}H\left(u\left(t\right)\right)\right)&:= \left.\frac{d}{d\varepsilon}\right|_{\varepsilon=0} H \left(u+\varepsilon v\right) \\
&= \left. \int \partial_x^j \left(u+ \varepsilon v\right) \cdot \partial_x^j v - \frac{1}{2} \left(u+\varepsilon v\right)^2 v dx \right|_{\varepsilon=0} \\
&= \int \partial_x^j u \cdot \partial_x^j v - \frac{1}{2}u^2v dx \\
&= \int \left[\left(-1\right)^j \partial_x^{2j} u - \frac{1}{2}u^2\right]v dx \\
&=  \int \left[\left(-1\right)^{j+1} \partial_x^{2j+1} u + \frac{1}{2}\partial_x\left(u^2\right)\right] \partial_x^{-1} v dx \\
&= \omega_{-\frac{1}{2}} \left(\left(-1\right)^{j+1}\partial_x^{2j+1} u+ \frac{1}{2} \partial_x \left(u^2\right),v\right) \\
&=\omega_{-\frac{1}{2}} \left(v,-\left(-1\right)^{j+1}\partial_x^{2j+1} u- \frac{1}{2} \partial_x \left(u^2\right)\right).
\end{align*}
Since the solution map of \eqref{eq:kdv} is a symplectic transform from $\HS$ to itself,  we can regard the function space and the solution map as the phase space and  the symplectic transform, respectively. Remark that the symplectic form does not depend on $j$, so we do not need to consider other symplectic forms or phase spaces for each $j$. With the obtained phase space and the symplectic transform, we state the second main theorem, the nonsqueezing property of \eqref{eq:kdv}.

\begin{theorem}[Nonsqueezing property: analytic version]\label{thm:Nonsqueezing thm}
Let $j \ge 2$, $ 0 < r< R$, $u_* \in \HS \left( \MB{T}\right)$, $k_0 \in \MB{Z}^*(=\MB{Z}\setminus \left\{0\right\})$, $z \in \MB{C}$ and $T>0$. Then there exists a global $\HS$-solution $u$ to \eqref{eq:kdv} such that
\begin{equation*}
\left\|u_0 - u_*\right\|_{\HS} \leq R
\end{equation*}
and 
\begin{equation*}
\left|k_0\right|^{-1/2} \left| \MC{F}_x\left(S_H\left(T\right)u_0\right) \left(k_0\right) - z\right| >r,
\end{equation*}
where $\MC{F}_x$ and $S_H$ are the spatial Fourier transform and the solution map of \eqref{eq:kdv}, respectively.\end{theorem}

The proof of Theorem \ref{thm:Nonsqueezing thm} follows arguments in \cite{Bourgain:1994tr} and \cite{CKSTT3}. In \cite{Bourgain:1994tr}, Bourgain proved the nonsqueezing property of the nonlinear Schr\"odinger equation on $L^2_x\left(\MB{T}\right)$ space. After taking the frequency truncation to the original equation, he applied Gromov's nonsqueezing theorem for the finite dimensional Hamiltonian system. From the approximation argument, the result is extended to the infinite dimensional NLS flow. Bourgain used basic (or a sharp) frequency truncation and $X^{s,b}$ space for this argument. 

Later, this argument extended by Colliander et al. \cite{CKSTT3} for the KdV flow on its phase space $H^{-1/2}_x\left(\MB{T}\right)$ with two more additional ingredients. Firstly, they found a counter example that the sharp truncated flow does not approximate to the original flow. Hence, they used a smooth truncation to resolve this problem. Secondly, they used the Miura transform to close the approximation argument. Indeed, they obtained the approximation result for the KdV equation by using the mKdV approximation result and the bi-continuity of the Miura transform. They proved approximation by truncated flow for mKdV flow and using the bi-continuity of the Miura transform in the some sense, concluded the approximation for the KdV flow. 

Like former results, our main tasks are also to find appropriate truncation and prove the approximation argument. We use the sharp truncation like Bourgain's approach. Even if (\ref{eq:kdv}) has the same symplectic phase space and the strength of the nonlinearity as in the KdV equation, much stronger modulation effect than that in the KdV equation facilitates that the finite dimensional system well approximates to the original infinite dimensional system without using the smooth truncation and the Miura transform. We note that with analytic version, the nonsqueezing property tells that the solution  flow does not transfer the energy between low and  high frequencies on the symplectic manifold, $\HS$. 

From now on, we consider a concrete truncated equation and other objects. Let $P_{\le N}$ be the Fourier projection for the spatial frequency as in \eqref{eq:Fourier multiplier}, we introduce the truncated equation
\begin{equation}\label{eq:truncated equation}
\left\{\begin{array}{ll}
\pt u + (-1)^{j+1} \px^{2j+1}u + P_{\le N}\left(\frac12 \px(u^2)\right)=0, \hspace{1em} &(t,x) \in \R \times \T, \\
u(x,0) = u_0(x), &u_0 \in P_{\leq N}H_0^s(\T).
\end{array}
\right.
\end{equation}
Denote the nonlinear flow of \eqref{eq:truncated equation} by $S_H^N(t)$. Using \eqref{eq:symplectic form}, we know that \eqref{eq:truncated equation} has the truncated Hamiltonian,
\begin{equation*}
H_N\left(u\left(t\right)\right) := \int_{\MB{T}} \frac{1}{2} \left(\partial_x^j u\right) ^2 - P_{\le N}\left(\frac{1}{6} u^3\right) dx.
\end{equation*}
Thus, this flow is the finite dimensional symplectic map, so we can apply Theorem \ref{thm:finite dimensional nonsq.} directly (see Lemma \ref{lem:Nonsqueezing of trun. fow}). Also, the equation (\ref{eq:truncated equation}) is locally and globally well-posed by using the similar argument as in \cite{Hirayama} and Section 3, respectively. In Section 4, we provide the proof of the approximation argument, and hence, we can completely obtain the nonsqueezing property of (\ref{eq:kdv}). 

We now restate Theorem \ref{thm:Nonsqueezing thm} geometrically for better understanding. To do this, we may define balls and cylinders. Let  $B^{\infty}_r\left(u_*\right)$ be an infinite dimensional ball in ${H^{-1/2}_0}$ of radius $r$ and centered at $u_* \in {H^{-1/2}_0}$ and  $C^{\infty}_{k,r}\left(z\right)$ be an infinite dimensional cylinder in ${H^{-1/2}_0}$ of radius $r$ and centered at $z \in \MB{C}$:
\begin{align*}
B^{\infty}_r\left(u_*\right) := \left\{u \in {H^{-1/2}_0} : \left\|u-u_*\right\|_{{H^{-1/2}_0}} \leq r\right\}, \\
C^{\infty}_{k,r}\left(z\right) := \left\{u \in {H^{-1/2}_0} : \left|k\right|^{-1/2}\left|\widehat{u}\left(k\right)-z\right| \leq r\right\}.
\end{align*}
The following is the geometric version of Theorem \ref{thm:Nonsqueezing thm} with respect to \eqref{eq:kdv}
\begin{theorem}[Nonsqueezing property: geometric version]
Let $0 < r< R$, $u_* \in \HS \left( \MB{T}\right)$, $k_0 \in \MB{Z}^*$, $z \in \MB{C}$ and $T>0$. Then
\begin{equation*}
\SK \left(T\right) \left( B^{\infty}_R \left(u_*\right)\right) \not \subseteq C^{\infty}_{k_0,r} \left(z\right),
\end{equation*}
where $S_H$ be the solution map of \eqref{eq:kdv} when $j>1$.
\end{theorem}

\subsection{Notations}
We clear some terminologies for our results. We use the spatial Fourier transform, the inverse Fourier transform and the space-time Fourier transform as follows:
\begin{align*}
\MC{F}_x \left(u\right) =\widehat{u}\left(k\right) &=  \int _{\MB{T}} e^{-ikx}u\left(x\right) dx, \\
u \left(x \right) &=  \frac{1}{2\pi}\int_{\Z}  e^{ikx}\widehat{u}\left(k\right) dk:=\frac{1}{2 \pi}\sum_{k \in \MB{Z}} \widehat{u}\left(k\right) e^{ikx},\\
\MC{F}\left(u\right) = \wt{u}\left(\tau,k\right)  &= \iint_{\MB{T} \times \MB{R}} e^{-ikx}e^{-i\tau t} u\left(x,t\right) dxdt.
\end{align*}
We have the spatial Sobolev space
\begin{equation}\label{eq:H^s space}
\left\|u\right\|_{H^s} = \left\|\left<k\right>^s \widehat{u}\right\|_{\ell^2_k} := \frac{1}{\left(2\pi\right)^{1/2}} \left(\sum_{k \in \MB{Z}} \left<k\right>^{2s}\left|\widehat{u}\right|^2\right)^{1/2}
\end{equation}
for $s \in \MB{R}$, where $\left<k\right> = \left(1+\left|k\right|^2\right)^{1/2}$. For each dyadic number $N$, we define the Fourier multipliers,
\begin{equation}\label{eq:Fourier multiplier}
\begin{split}
\widehat{P_Nu}\left(k\right) := 1_{N \leq \left|k\right| <2N}\left(k\right) \widehat{u}\left(k\right), \\
\widehat{P_{\leq N} u}\left(k \right) := 1_{\left|k\right| \leq N}\left(k\right) \widehat{u}\left(k\right), \\
\widehat{P_{\geq N} u}\left(k \right) := 1_{\left|k\right| \geq N}\left(k\right) \widehat{u}\left(k\right),
\end{split}
\end{equation}
where $1_{\Omega}$ is a characteristic function on $\Omega$.  By the mean preserving \eqref{eq:conservation laws} and the Galilean transform, we have the mean zero function space with the same norm as in (\ref{eq:H^s space}) as follows:
\begin{equation*}
H^s_0 = \left\{u \in H^s :  \int_{\MB{T}} u =0 \right\}.
\end{equation*}

We define the general $X^{s,b}$ norm associated to \eqref{eq:kdv},
\begin{equation*}\label{eq:X^sb norm}
\left\|u\right\|_{X^{s,b}} = \left\|\left<k\right>^s\left<\tau-k^{2j+1}\right>^b \widetilde{u} \right\|_{L^2_{\tau}\ell^2_k}.
\end{equation*}
Using this, we define $Y^s$ and $Z^s$ spaces for the solution and nonlinear term under the norms
\[\norm{f}_{Y^s} = \norm{f}_{X^{s,\frac12}} + \norm{\bra{k}^s\wt{f}}_{\ell_k^2L_{\tau}^1},\]
\[\norm{f}_{Z^s} = \norm{f}_{X^{s,-\frac12}} + \norm{\bra{k}^s\bra{\tau - k^{2j+1}}^{-1}\wt{f}}_{\ell_k^2L_{\tau}^1}.\]

For $x, y \in \R_+$, $x \lesssim y$ denotes $x \le Cy$ for some $C >0$ and $x \sim y$ means $x \lesssim y$ and $y \lesssim x$. Using this, we denote $f = O(g)$ by $f \lesssim g$ for positive real-valued functions $f$ and $g$. Moreover, $x \ll y$ denotes $x \le cy$ for small positive constant $c$. Let $a_1,a_2,a_3 \in \R$ and $b_1,b_2,b_3,b_4 \in \R$. The quantities $a_{max} \ge a_{med} \ge a_{min}$ can be defined to be the maximum, median and minimum values of $a_1,a_2,a_3$, respectively, Also, $b_{max} \ge b_{sub} \ge b_{thd} \ge b_{min}$ can be defined similarly as before.

The paper is organized as follows: In Section \ref{sec:bi-,tri-}, we give algebraic results for the resonant functions, and prove the bi- and trilinear estimates for the global well-posedness and the nonsqueezing property. In Section \ref{sec:global}, we prove the global well-posedness of \eqref{eq:kdv}. In Section \ref{sec:nonsqueezing}, we prove the nonsqueezing property of the solution flow of \eqref{eq:kdv} by showing the approximation argument between the original and truncated flows.

\textbf{Acknowledgement.} The authors would like to thank their advisor Soonsik Kwon for his helpful comments and encouragement through this research problem. The authors are partially supported by NRF (Korea) grant 2015R1D1A1A01058832.

\section{Bi- and Trilinear estimates}\label{sec:bi-,tri-}
In this section, we will prove some algebraic analysis, bi- and trilinear estimates which are useful tools to prove the global well-posedness and nonsqueezing property of \eqref{eq:kdv} in section \ref{sec:global} and \ref{sec:nonsqueezing}, respectively. We first observe some algebraic analysis results.

\begin{lemma}\label{lem:algebra}
Let $j \in \N$.

{\rm{(a)}} If $x,y,z \in \R$ with $x+y+z=0$. Then we have 
\begin{equation}\label{eq:algebra1}
P_3(x,y,z) = x^{2j+1} + y^{2j+1} + z^{2j+1} = xyz\cdot Q_3(x,y,z),
\end{equation}
where $|Q_3(x,y,z)| \sim \max(|x|,|y|,|z|)^{2j-2}$.

{\rm{(b)}} If $x,y,z,w \in \R$ with $x+y+z+w=0$. Then we have 
\begin{equation}\label{eq:algebra2}
P_4(x,y,z,w) = x^{2j+1} + y^{2j+1} + z^{2j+1} +w^{2j+1}= (x+y)(x+z)(x+w)\cdot Q_4(x,y,z,w),
\end{equation}
where $|Q_4(x,y,z,w)| \sim \max(|x|,|y|,|z|,|w|)^{2j-2}$.
\end{lemma}
\begin{proof}
(a) can be obtained by the similar argument for (b) (or see \cite{Hirayama}). Hence, we only prove the second part of Lemma \ref{lem:algebra}. We may assume that $|x| \ge |y| \ge |z| \ge |w|$ without loss of generality. If at least one of $x+y$, $x+z$ and $x+w$ is zero, we can easily see $x^{2j+1} + y^{2j+1} + z^{2j+1} +w^{2j+1} = 0$ and thus it suffices to show $|Q_4(x,y,z,w)| \sim \max(|x|,|y|,|z|,|w|)^{2j-2}$. 

\textbf{Case I.} $|x|\sim|y|\gg|z|$. From $x+y+z+w =0$, we may assume that $x>0$ and $-y >0$. Then, \eqref{eq:algebra2} is equivalent that for $x,y >0$, 
\[x^{2j+1} - y^{2j+1} + z^{2j+1} +w^{2j+1}= (x-y)(x+z)(x+w)\cdot Q_4'(x,y,z,w),\]
where $x-y+z+w=0$ and $|Q_4| = |Q_4'|$. By the mean value theorem (MVT), we have
\begin{equation}\label{eq:calculation1}
\frac{x^{2j+1} - y^{2j+1} + z^{2j+1} +w^{2j+1}}{x-y} = (x^{\ast})^{2j} + \frac{z^{2j+1} +w^{2j+1}}{x-y},
\end{equation}
for some $y < x^{\ast} < x$. For the rest term of the right-hand side of \eqref{eq:calculation1}, from the following identity
\begin{align*}
\frac{z^{2j+1} +w^{2j+1}}{x-y}&= - \frac{z^{2j+1} +w^{2j+1}}{z+w} \\
&=- (z^{2j} - z^{2j-1}w + \cdots -zw^{2j-1}+w^{2j}), 
\end{align*}
and $|z| \ll |x|$, we have
\[\left| \frac{z^{2j+1} +w^{2j+1}}{x-y} \right| \ll |x|^{2j}.\]
Hence, we conclude that
\[\left|\frac{x^{2j+1} - y^{2j+1} + z^{2j+1} +w^{2j+1}}{x-y} \right| \sim |x|^{2j},\]
which implies $|Q_4'| \sim |x|^{2j-2}$ from $|x+z|,|x+w| \sim |x|$.

\textbf{Case II.} $|x| \sim |z| \gg |w|$. From $x+y+z+w =0$, we may assume that $x>0$ and $-y,-z >0$. Moreover, we have $|x+y|,|x+z| \sim |x|$ and $|x^{2j+1} + y^{2j+1} + z^{2j+1} +w^{2j+1}| \sim |x|^{2j+1}$ , and thus $|Q_4| \sim |x|^{2j-2}$.

\textbf{Case III.} $|x|\sim|w|$. We may assume that $x,w >0$ and $-y,-z>0$. Then, similarly as before, \eqref{eq:algebra2} is equivalent that for $x,y,z,w >0$, 
\[x^{2j+1} - y^{2j+1} - z^{2j+1} +w^{2j+1}= (x-y)(x-z)(x+w)\cdot Q_4'(x,y,z,w),\]
where $x-y-z+w=0$ and $|Q_4| = |Q_4'|$. Using the MVT twice, we can obtain that
\begin{align*}
\frac{x^{2j+1} - y^{2j+1} - z^{2j+1} +w^{2j+1}}{(x-y)(x-z)} &= \frac{(x^{\ast})^{2j} - (z^{\ast})^{2j}}{x-z}, \hspace{1.5em} y < x^{\ast} < x, ~ w < z^{\ast} < z\\
&= (x^{\ast\ast})^{2j-1}, \hspace{1.5em} z^{\ast} < x^{\ast\ast} < x^{\ast}. 
\end{align*} 
Hence, we conclude from $|x+w| \sim |x|$ that $|Q_4'| \sim |x|^{2j-2}$.
\end{proof}

Now, we state the $L^4$-Strichartz estimate which is a useful tool to prove bi- and trilinear estimates.
\begin{lemma}\label{lem:strichartz}
Let $j \in \MB{N}$. For any function $u \in \T \times \R$, we have the $L^4$-Strichartz estimate for \eqref{eq:kdv}
\begin{equation}\label{eq:strichartz}
\norm{u}_{L_{t,x}^4} \lesssim \norm{u}_{X^{0,\frac{j+1}{2(2j+1)}}}.
\end{equation}
In particular, we have $\norm{u}_{L_{t,x}^4} \lesssim \norm{u}_{X^{0,\frac13}}$.
\end{lemma}

\begin{proof}
This type estimate was first introduced by Bourgain \cite{Bourgain1993} associated to the Schr\"odinger and the KdV equations. Moreover, one can also find the comment for Lemma \ref{lem:strichartz} in \cite{Bourgain1995}. The proof of this lemma is almost similar as in \cite{Bourgain1993} and hence, we omit the detailed proof. We also refer \cite{Tao1} and \cite{Tao2} for the proof.
\end{proof}

From now on, let us consider the bi- and trilinear estimates which are the main results in this section. We already know the bilinear estimate proved by Hirayama in \cite{Hirayama} as follows:
\begin{proposition}[Hirayama \cite{Hirayama}]
Let $j \in \N$ and $s \ge -j/2$. Then, the following bilinear estimate holds:
\begin{equation}\label{eq:bilinear1}
\norm{\ft^{-1}[\bra{\tau - k^{2j+1}}^{-1}\wt{\px uv}]}_{Z^s} \lesssim \norm{u}_{Y^s}\norm{v}_{Y^s}.
\end{equation}
\end{proposition}
However, for our analysis, refined estimates of \eqref{eq:bilinear1} are needed. The following lemma will be used to prove the global well-posedness.
\begin{lemma}\label{lem:bilinear2}
Let $j \in \N$ and $s \ge -j/2$. Let $u_i = P_{N_i}u$ and $|k_i|\sim N_i\ge1$, $i=1,2,3$. Then we have 
\begin{equation}\label{eq:bilinear2}
\norm{P_{N_3}\px(u_1u_2)}_{X^{s,-\frac12}} \lesssim (N_1N_2)^{-\frac12}N_3^{s+\frac12}N_{max}^{1-j}\norm{u_1}_{X^{0,\frac12}}\norm{u_2}_{X^{0,\frac12}}.
\end{equation}
\end{lemma}
\begin{proof}
Due to the total derivative in the left-hand side of \eqref{eq:bilinear2}, we may assume that $k_3 \neq 0$. Let $\lambda_i = \tau_i -k_i^{2j+1}$, $i=1,2,3$. Then, from the definition of $X^{s,b}$-norm, we can reduce \eqref{eq:bilinear2} by  
\begin{equation}\label{eq:equiv}
\begin{aligned}
&\normo{\sum_{\substack{k_1,k_2\in\Z^{\ast}\\k_1+k_2=k_3}}\int_{\tau_1 + \tau_2 = \tau_3}\frac{|k_3|\bra{k_3}^s|k_3|^{-s-\frac12}|k_1k_2|^{\frac12}N_{max}^{j-1}}{\prod_{i=1}^{3}\bra{\lambda_i}^{\frac12}}\wt{u}_1(k_1,\tau_1)\wt{u}_2(k_2,\tau_2)\: d\tau_1}_{\ell_{k_3}^2L_{\tau_3}^2}\\
&\hspace{28em}\lesssim \norm{u_1}_{L_{t,x}^2}\norm{u_2}_{L_{t,x}^2}.
\end{aligned}
\end{equation}
Without loss of generality, we assume $|\lambda_1|\le|\lambda_2|\le|\lambda_3|$. Then, from the identity
\[\lambda_1 + \lambda_2 = \lambda_3 - P_3(k_1,k_2,-k_3),\]
where $P_3$ is defined as in \eqref{eq:algebra1}, and Lemma \ref{lem:algebra} (a), we have $|\lambda_3| \gtrsim |k_1k_2k_3|\max(|k_1|,|k_2|,|k_3|)^{2j-2}$. By using this and duality, the left-hand side of \eqref{eq:equiv} is dominated by
\begin{equation}\label{eq:integral}
\int_{\R \times \T} u_3\ft^{-1}[\bra{\tau_1-k_1^{2j+1}}^{-\frac12}\wt{u}_1]\ft^{-1}[\bra{\tau_2-k_2^{2j+1}}^{-\frac12}\wt{u}_2]\:dxdt,
\end{equation}
where $\norm{u_3}_{L_{t,x}^2} \le 1$. We apply the H\"older inequality and Lemma \ref{lem:strichartz} ($X^{0,\frac13} \subset L_{t,x}^4$) to \eqref{eq:integral}, then \eqref{eq:integral} is bounded by
\begin{align*}
&\norm{u_3}_{L_{t,x}^2}\norm{\ft^{-1}[\bra{\tau_1-k_1^{2j+1}}^{-\frac12}\wt{u}_1]}_{L_{t,x}^4}\norm{\ft^{-1}[\bra{\tau_2-k_2^{2j+1}}^{-\frac12}\wt{u}_2}_{L_{t,x}^4} \\
\lesssim& \norm{u_3}_{L_{t,x}^2}\norm{u_1}_{X^{0,-\frac16}}\norm{u_2}_{X^{0,-\frac16}}\\
\lesssim& \norm{u_1}_{L_{t,x}^2}\norm{u_2}_{L_{t,x}^2},
\end{align*}
and this completes the proof.
\end{proof}

The following lemma will be used to obtain the nonsqueezing property.
\begin{lemma}
Let $j \in \N$. Let $u_i = P_{N_i}u$ and $|k_i|\sim N_i \ge 1$, $i=1,2,3$. Then we have 
\begin{equation}\label{eq:bilinear3}
\norm{P_{N_3}\px(u_1u_2)}_{Z^{-\frac12}} \lesssim N_{max}^{1-j}\norm{u_1}_{Y^{-\frac12}}\norm{u_2}_{Y^{-\frac12}}
\end{equation}
\end{lemma}
\begin{remark}
Thanks to the frequency decay bound $N_{max}^{1-j}$, $j > 1$, one can easily obtain an error bound in the proof of Lemma \ref{prop:approx1}, and this guarantees the approximation of the higher-order KdV flow without the Miura transform. 
\end{remark}
\begin{proof}
We also assume $k_3 \neq 0$ due to the same reason in the proof of Lemma \ref{lem:bilinear2}.

We first control the $\ell_k^1L_{\tau}^2$ part. From the definition of function spaces, it suffices to show that
\begin{equation}\label{eq:bi1}
\normo{\sum_{\substack{k_1,k_2 \in \Z^{\ast}\\k_1+k_2=k_3}}\int_{\tau_1 + \tau_2 = \tau_3}\frac{|k_1k_2k_3|^{\frac12}\wt{u}_1(k_1,\tau_1)\wt{u}_2(k_2,\tau_2)}{\bra{\lambda_3}\bra{\lambda_1}^{\frac12}\bra{\lambda_2}^{\frac12}}\: d\tau_1}_{\ell_{k_3}^2L_{\tau_3}^1} \lesssim N_{max}^{1-j}\norm{u_1}_{X^{0,0}}\norm{u_2}_{X^{0,0}},
\end{equation}
where $\lambda_i$ is defined in the proof of Lemma \ref{lem:bilinear2}. From \eqref{eq:algebra1}, we have $\max(|\lambda_1|,|\lambda_2|,|\lambda_3|) \gtrsim |k_1k_2k_3|\max(|k_1|,|k_2|,|k_3|)^{2j-2}$. If $|\lambda_1|=\max(|\lambda_1|,|\lambda_2|,|\lambda_3|)$, \eqref{eq:bi1} is reduced by
\begin{equation}\label{eq:bi2}
\normo{\sum_{\substack{k_1,k_2 \in \Z^{\ast}\\k_1+k_2=k_3}}\int_{\tau_1 + \tau_2 = \tau_3}\bra{\lambda_3}^{-1}\wt{u}_1(k_1,\tau_1)\bra{\lambda_2}^{-\frac12}\wt{u}_2(k_2,\tau_2)\: d\tau_1}_{\ell_{k_3}^2L_{\tau_3}^1} \lesssim \norm{u_1}_{L_{t,x}^2}\norm{u_2}_{L_{t,x}^2}.
\end{equation}
Since $\bra{\lambda_3}^{-\frac23}$ is $L^2$-integrable, by the Cauchy-Schwarz inequality with respect to $\tau_3$, the left-hand side of \eqref{eq:bi2} is bounded by
\begin{equation*}
\normo{\sum_{\substack{k_1,k_2 \in \Z^{\ast}\\k_1+k_2=k_3}}\int_{\tau_1 + \tau_2 = \tau_3}\bra{\lambda_3}^{-\frac13}\wt{u}_1(k_1,\tau_1)\bra{\lambda_2}^{-\frac12}\wt{u}_2(k_2,\tau_2)\: d\tau_1}_{\ell_{k_3}^2L_{\tau_3}^2}.
\end{equation*}
Then, by duality and $X^{0,\frac13} \subset L_{t,x}^4$, we can obtain
\begin{align*}
\int_{\R \times \T} \ft^{-1}[\bra{\lambda_3}^{-\frac13}\wt{u}_3]&u_1\ft^{-1}[\bra{\lambda_2}^{-\frac12}\wt{u}_2]\: dxdt \\
&\lesssim \norm{\ft^{-1}[\bra{\lambda_3}^{-\frac13}\wt{u}_3]}_{L_{t,x}^4}\norm{u_1}_{L_{t,x}^2}\norm{\ft^{-1}[\bra{\lambda_2}^{-\frac12}\wt{u}_2]}_{L_{t,x}^4} \\
&\lesssim \norm{u_3}_{X^{0,0}}\norm{u_1}_{X^{0,0}}\norm{u_2}_{X^{0,-\frac16}},
\end{align*}
where $\norm{u_3}_{L_{t,x}^2} \le 1$. The last term implies the right-hand side of \eqref{eq:bi2}. By symmetry, we do not need to consider $|\lambda_2| = \max(|\lambda_1|,|\lambda_2|,|\lambda_3|)$. Next, we consider the case when $|\lambda_3| = \max(|\lambda_1|,|\lambda_2|,|\lambda_3|)$. If $\bra{\lambda_1} \gtrsim |P_3(k_1,k_2,-k_3)|^{\frac{1}{100}}$, we can reduce the left-hand side of \eqref{eq:bi1} as
\[\normo{\sum_{\substack{k_1,k_2 \in \Z^{\ast}\\k_1+k_2=k_3}}\int_{\tau_1 + \tau_2 = \tau_3}\bra{\lambda_3}^{-\frac12-\frac{1}{600}}N_{max}^{1-j}\bra{\lambda_1}^{-\frac13}\wt{u}_1(k_1,\tau_1)\bra{\lambda_2}^{-\frac12}\wt{u}_2(k_2,\tau_2)\: d\tau_1}_{\ell_{k_3}^2L_{\tau_3}^1}.\]
Since $\bra{\lambda_3}^{-\frac12-\frac{1}{600}}$ is $L^2$-integrable, using the Cauchy-Schwarz inequality with respect to $\tau_3$, duality, the H\"older inequality and $X^{0,\frac13} \subset L_{t,x}^4$, we have the \eqref{eq:bi1}. Finally, we assume $\bra{\lambda_i} \ll |P_3(k_1,k_2,-k_3)|^{\frac{1}{100}}$ for $i=1,2$. Then the left-hand side of \eqref{eq:bi1} is bounded by
\[\normo{\sum_{\substack{k_1,k_2 \in \Z^{\ast}\\k_1+k_2=k_3}} |P_3|^{-\frac12}N_{max}^{1-j}\int_{\bra{\lambda_1} \ll |P_3|^{\frac{1}{100}}}\int_{\bra{\lambda_2} \ll |P_3|^{\frac{1}{100}}}\wt{u}_1(k_1,\tau_1)\wt{u}_2(k_2,\tau_2)\:d\tau_1d\tau_2}_{\ell_{k_3}^2}.\]
Applying the Cauchy-Schwarz inequality in $\tau_1$ and $\tau_2$ separately, above norm is dominated by
\[\normo{\sum_{\substack{k_1,k_2 \in \Z^{\ast}\\k_1+k_2=k_3}} |P_3|^{-\frac12+\frac{1}{100}}N_{max}^{1-j}F_1(k_1)F_2(k_2)}_{\ell_{k_3}^2},\]
where 
\[F_i(k_i) = \norm{\wt{u}_i(\tau_i)}_{L_{\tau_i}^2}.\]
From \eqref{eq:algebra1}, we have $|P_3| \gtrsim |k_i|^{2j}$ for $i=1,2,3$, which implies
\[|P_3|^{-\frac12+\frac{1}{100}} \lesssim |k_3|^{-j+\frac{j}{50}},\]
and this should be $\ell^2$-summable in $k_3$. Hence, we finally have
\[\normo{\sum_{k_1 \in \Z^{\ast}}N_{max}^{1-j}F_1(k_1)F_2(k_3-k_1)}_{\ell_{k_3}^{\infty}},\]
this is bounded from the Cauchy-Schwarz inequality in $k_1$ that
\[N_{max}^{1-j}\norm{F_1}_{\ell_{k_1}^2}\norm{F_2}_{\ell_{k_2}^2} = N_{max}^{1-j}\norm{u_1}_{X^{0,0}}\norm{u_2}_{X^{0,0}}.\]
For the $X^{-\frac12,-\frac12}$ part, it follows directly from Lemma \ref{lem:bilinear2}, when $s=-\frac12$. Hence we complete the proof.
\end{proof} 
The following trilinear estimates will be also helpful to prove the global well-posedness.

\begin{lemma}\label{lem:trilinear}
Let $j \in \N$ and $-j/2 \le s < 0$. Let $u_i = P_{N_i}u$ and $|k_i|\sim N_i \ge 1$, $i=1,2,3$. Suppose that $k=k_1+k_2+k_3$, $|k_1|\ge|k_2|\ge|k_3|$ and $P_4(k_1,k_2,k_3,-k) \neq 0$, where $P_4$ is defined as in Lemma \ref{lem:algebra}. 

{\rm{(a)}} If $|k|\sim|k_1|$, then
\begin{equation}\label{eq:trilinear1}
\norm{u_1u_2u_3}_{X^{-s,\frac12}} \lesssim N_1^{-s+j}N_2^{\frac12}N_3^{\frac12}\norm{u_1}_{Y^0}\norm{u_2}_{Y^0}\norm{u_3}_{Y^0}
\end{equation}

{\rm{(b)}} If $|k_1|\sim|k_2|$ and $j \ge 2$, then
\begin{equation}\label{eq:trilinear2}
\norm{u_1u_2u_3}_{X^{-s-j,\frac12}} \lesssim N_1^jN_3\norm{u_1}_{Y^0}\norm{u_2}_{Y^0}\norm{u_3}_{Y^0}, \hspace{2em} \mbox{for} \hspace{1em} |k_3| \ge |k|,
\end{equation}
or
\begin{equation}\label{eq:trilinear3}
\norm{u_1u_2u_3}_{X^{-s-j-\frac12,\frac12}} \lesssim N_1^jN_3^{\frac12}\norm{u_1}_{Y^0}\norm{u_2}_{Y^0}\norm{u_3}_{Y^0}, \hspace{2em} \mbox{for}  \hspace{1em} |k| \ge |k_3|.
\end{equation}
\end{lemma}

\begin{proof}
From the Plancherel theorem, we have the identity
\[\norm{u_1u_2u_3}_{X^{-s,\frac12}} = \normo{\bra{k}^{-s}\bra{\tau-k^{2j+1}}^{\frac12}\sum_{\substack{k_1,k_2,k_3 \in \Z^{\ast}\\k_1+k_2+k_3=k}}\int_{\tau_1 + \tau_2 + \tau_3 = \tau}\prod_{i=1}^3\wt{u}_i(\tau_i,k_i)\: d\tau_1d\tau_2}_{\ell_k^2L_{\tau}^2}.\]

{\rm{(a)}} We first consider $\bra{\tau - k^{2j+1}} \lesssim \bra{\tau_i - k_i^{2j+1}}$ for some $i=1,2,3$. We may assume that $\bra{\tau - k^{2j+1}} \lesssim \bra{\tau_1 - k_1^{2j+1}}$. Then, \eqref{eq:trilinear1} is restricted by
\begin{equation}\label{eq:tril1}
\norm{u_1u_2u_3}_{X^{-s,0}} \lesssim N_1^{-s+j}N_2^{\frac12}N_3^{\frac12} \norm{u_1}_{X^{0,0}}\norm{u_2}_{Y^0}\norm{u_3}_{Y^0}.
\end{equation}
Since $|k|\sim |k_1|$, from the Young's and the Cauchy-Schwarz inequalities, we have
\begin{align*}
\mbox{LHS of } \eqref{eq:tril1} &\sim N_1^{-s}\normo{\sum_{\substack{k_1,k_2,k_3 \in \Z^{\ast}\\k_1+k_2+k_3=k}}\int_{\tau_1 + \tau_2 + \tau_3 = \tau}\prod_{i=1}^3\wt{u}_i(\tau_i,k_i)\: d\tau_1d\tau_2}_{\ell_k^2L_{\tau}^2}\\
&\lesssim N_1^{-s}\norm{\wt{u}_1}_{\ell_k^2L_{\tau}^2}\norm{\wt{u}_2}_{\ell_k^1L_{\tau}^1}\norm{\wt{u}_3}_{\ell_k^1L_{\tau}^1}\\
&\lesssim N_1^{-s}N_2^{\frac12}N_3^{\frac12}\norm{u_1}_{X^{0,0}}\norm{\wt{u}_2}_{\ell_k^2L_{\tau}^1}\norm{\wt{u}_3}_{\ell_k^2L_{\tau}^1},
\end{align*}
the last term implies the right-hand side of \eqref{eq:trilinear1}.

Now, we consider $\bra{\tau - k^{2j+1}} \gg \bra{\tau_i - k_i^{2j+1}}$ for all $i=1,2,3$. Then, from \eqref{eq:algebra2}, we have
\begin{align*}
|\tau - k^{2j+1}| &\sim |P_4(k_1,k_2,k_3,-k)|\\
&\sim |(k-k_1)(k-k_2)(k-k_3)||k_1|^{2j-2}\\
&\lesssim |k_1|^{2j}|k_2|.
\end{align*}
Hence,
\begin{align*}
\norm{u_1u_2u_3}_{X^{-s,\frac12}} &\lesssim N_1^{-s+j}N_2^{\frac12}\norm{u_1u_2u_3}_{L_{t,x}^2} \\
&\lesssim N_1^{-s+j}N_2^{\frac12}\norm{u_1u_2}_{L_{t,x}^2}\norm{u_3}_{L_{t,x}^{\infty}}.
\end{align*}
From the H\"older and the Sobolev inequalities, \eqref{eq:strichartz} and $Y^s \subset C_t H^s$ we have
\begin{align*}
\norm{u_1u_2}_{L_{t,x}^2}\norm{u_3}_{L_{t,x}^{\infty}} &\lesssim N_3^{\frac12}\norm{u_1}_{L_{t,x}^4}\norm{u_2}_{L_{t,x}^4}\norm{u_3}_{L_t^{\infty}L_x^2} \\
&\lesssim N_3^{\frac12}\norm{u_1}_{X^{0,\frac13}}\norm{u_2}_{X^{0,\frac13}}\norm{u_3}_{Y^0}
\end{align*}
and this implies the right-hand side of \eqref{eq:trilinear1}.

{\rm{(b)}} We consider firstly $\bra{\tau - k^{2j+1}} \lesssim \bra{\tau_1 - k_1^{2j+1}}$ similarly as in (a). Then, \eqref{eq:trilinear2} and \eqref{eq:trilinear3} are also restricted by
\begin{equation}\label{eq:tril3}
\norm{u_1u_2u_3}_{X^{-s-j,0}} \lesssim N_1^jN_3 \norm{u_1}_{X^{0,0}}\norm{u_2}_{Y^0}\norm{u_3}_{Y^0},
\end{equation}
and
\begin{equation}\label{eq:tril4}
\norm{u_1u_2u_3}_{X^{-s-j-\frac12,0}} \lesssim N_1^jN_3^{\frac12} \norm{u_1}_{X^{0,0}}\norm{u_2}_{Y^0}\norm{u_3}_{Y^0}.
\end{equation}
Since $j \ge 2$, both $\bra{k}^{-s-j}$ and $\bra{k}^{-s-j-\frac12}$ are $\ell_k^2$-summable, and from the H\"older and the Young's inequalities, we obtain
\begin{align*}
\mbox{LHS of } \eqref{eq:tril3} &= \normo{\bra{k}^{-s-j}\sum_{\substack{k_1,k_2,k_3 \in \Z^{\ast}\\k_1+k_2+k_3=k}}\int_{\tau_1+\tau_2+\tau_3=\tau}\prod_{i=1}^3\wt{u}_i(\tau_i,k_i)\:d\tau_1d\tau_2}_{\ell_k^2L_{\tau}^2}\\
&\lesssim \normo{\sum_{\substack{k_1,k_2,k_3 \in \Z^{\ast}\\k_1+k_2+k_3=k}}\int_{\tau_1+\tau_2+\tau_3=\tau}\prod_{i=1}^3\wt{u}_i(\tau_i,k_i)\:d\tau_1d\tau_2}_{\ell_k^{\infty}L_{\tau}^2}\\
&\lesssim \norm{\wt{u}_1}_{\ell_k^2L_{\tau}^2}\norm{\wt{u}_2}_{\ell_k^2L_{\tau}^1}\norm{\wt{u}_3}_{\ell_k^1L_{\tau}^1}\\
&\lesssim N_3^{\frac12}\norm{u_1}_{X^{0,0}}\norm{\wt{u}_2}_{Y^0}\norm{\wt{u}_3}_{\ell_k^2L_{\tau}^1},
\end{align*}
and
\begin{align*}
\mbox{LHS of } \eqref{eq:tril4} &= \normo{\bra{k}^{-s-j-\frac12}\sum_{\substack{k_1,k_2,k_3 \in \Z^{\ast}\\k_1+k_2+k_3=k}}\int_{\tau_1+\tau_2+\tau_3=\tau}\prod_{i=1}^3\wt{u}_i(\tau_i,k_i)\:d\tau_1d\tau_2}_{\ell_k^2L_{\tau}^2}\\
&\lesssim \normo{\sum_{\substack{k_1,k_2,k_3 \in \Z^{\ast}\\k_1+k_2+k_3=k}}\int_{\tau_1+\tau_2+\tau_3=\tau}\prod_{i=1}^3\wt{u}_i(\tau_i,k_i)\:d\tau_1d\tau_2}_{\ell_k^{\infty}L_{\tau}^2}\\
&\lesssim \norm{\wt{u}_1}_{\ell_k^2L_{\tau}^2}\norm{\wt{u}_2}_{\ell_k^2L_{\tau}^1}\norm{\wt{u}_3}_{\ell_k^1L_{\tau}^1}\\
&\lesssim N_3^{\frac12}\norm{u_1}_{X^{0,0}}\norm{\wt{u}_2}_{Y^0}\norm{\wt{u}_3}_{\ell_k^2L_{\tau}^1},
\end{align*}
each last term implies the right-hand side of \eqref{eq:trilinear2} and \eqref{eq:trilinear3}, respectively.

Next, we consider $\bra{\tau - k^{2j+1}} \gg \bra{\tau_i - k_i^{2j+1}}$ for all $i=1,2,3$. Then, from \eqref{eq:algebra2}, we have similarly as before that
\begin{align*}
|\tau - k^{2j+1}| &\sim |P_4(k_1,k_2,k_3,-k)|\\
&\sim |(k-k_1)(k-k_2)(k-k_3)||k_1|^{2j-2}\\
&\lesssim |k_1|^{2j}\max(|k_3|,|k|).
\end{align*}
For $|k_3| \ge |k|$, since $\bra{k}^{-s-j}$ is $\ell^2$-summable, we use the similar argument as above to show
\begin{align*}
\norm{u_1u_2u_3}_{X^{-s-j,\frac12}} &\lesssim N_1^{j}N_3^{\frac12}\norm{\wt{u}_1 \ast \wt{u}_2 \ast \wt{u}_3}_{\ell_k^{\infty}L_{\tau}^2} \\
&\lesssim N_1^{j}N_3^{\frac12}\norm{\wt{u}_1}_{\ell_k^2L_{\tau}^2}\norm{\wt{u}_2}_{\ell_k^2L_{\tau}^1}\norm{\wt{u}_3}_{\ell_k^1L_{\tau}^1}\\
&\lesssim N_1^{j}N_3\norm{u_1}_{X^{0,0}}\norm{\wt{u}_2}_{Y^0}\norm{\wt{u}_3}_{\ell_k^2L_{\tau}^1}.
\end{align*}
For $|k| \ge |k_3|$, by the same argument, we have
\begin{align*}
\norm{u_1u_2u_3}_{X^{-s-j-\frac12,\frac12}} &\lesssim N_1^{j}\norm{\wt{u}_1 \ast \wt{u}_2 \ast \wt{u}_3}_{\ell_k^{\infty}L_{\tau}^2} \\
&\lesssim N_1^{j}\norm{\wt{u}_1}_{\ell_k^2L_{\tau}^2}\norm{\wt{u}_2}_{\ell_k^2L_{\tau}^1}\norm{\wt{u}_3}_{\ell_k^1L_{\tau}^1}\\
&\lesssim N_1^{j}N_3^{\frac12}\norm{u_1}_{X^{0,0}}\norm{\wt{u}_2}_{Y^0}\norm{\wt{u}_3}_{\ell_k^2L_{\tau}^1}.
\end{align*}
Thus, we complete the proof.
\end{proof}

\section{Global well-posedness for $j \ge 3$.}\label{sec:global}
In this section, we will prove the global well-posedness of \eqref{eq:kdv} for $-\frac{j}{2} \le s < 0$$\footnote{Due to the $L^2$-conservation laws, it suffices to consider the case when $-\frac{j}{2} \le s < 0$.}$, when $j \ge 3$. We use the method of almost conservation law (so called "I-method") in \cite{CKSTT1}. Before introducing the modified energy, we introduce some definitions.
\begin{definition}
An $n$-multiplier is a function $m : \R^n \to \mathbb{C}$. We say an $n$-multiplier $m$ is symmetric if
\[m(\xi_1, \cdots, \xi_n) = m(\xi_{\sigma(1)},\cdots \xi_{\sigma(n)}) \hspace{2em} \mbox{for all} \hspace{1em} \sigma \in S_n,\]
where $S_n$ is the group of all permutations on $n$ objects, with the symmetrization
\[[m(\xi_1,\cdots,\xi_n)]_{sym}:= \frac{1}{n!}\sum_{\sigma \in S_n}m(\xi_{\sigma(1)},\cdots \xi_{\sigma(n)}).\]
\end{definition}
Even though the domain of $m$ is $\R^n$, we will only be interested in $m$ on the hyperplane $\xi_1 + \cdots + \xi_n = 0$.
\begin{definition}
An $n$-linear functional $\Lambda _n$ acting on functions $u_1,\cdots,u_n$ generated by an $n$-multiplier $m$ is given by
\[\Lambda_n(m:u_1,\cdots,u_n) := \int_{\Gamma_n}m(\xi_1, \cdots, \xi_n) \wh{u}_1(\xi_1) \cdots \wh{u}_n(\xi_n),\]
where $\Gamma_n = \set{(\xi_1 \cdots,\xi_n) \in \R^n : \xi_1 + \cdots + \xi_n =0}$. In particular, when $u_1, \cdots, u_n$ are the same functions, we write $\Lambda_n(m)$.
\end{definition}
Now, let us define an operator $I$ which operates $\wh{Iu} = m(\xi)\wh{u}(\xi)$, and acts like an identity and an integral operator on low and high frequencies, respectively, by choosing a smooth monotone multiplier satisfying
\begin{equation}\label{eq:multiplier}
\begin{split}
m(\xi) = \left\{
\begin{array}{cl}
1, & |\xi| < N \\
N^{-s}|\xi|^s, & |\xi| > 2N ,
\end{array}
\right.
\end{split}
\end{equation}
for fixed $N$ (which will be chosen later). 

Let us define the modified energy $E_I^2(t)$ by
\[E_I^2(t) = \norm{Iu}_{L_x^2}^2 = \Lambda_2(m(\xi_1)m(\xi_2)).\]
The last equality follows from the Plancherel theorem and the facts that $u$ is real-valued, $m$ is even. In order to approach our goal, we further define modified energies (so called, \emph{correction terms}) by using the following lemma:
\begin{lemma}
Suppose $u$ be a solution of \eqref{eq:kdv} and $m$ is a symmetric $n$-multiplier. Then
\begin{equation}\label{eq:normal}
\frac{d}{dt}\Lambda_n(m) = \Lambda_n(m\alpha_n) -i\frac{n}{2}\Lambda_{n+1}([m(\xi_1, \cdots, \xi_{n-1}, \xi_n+\xi_{n+1})\set{\xi_n+\xi_{n+1}}]_{sym}),
\end{equation}
where
\[\alpha_n = i(\xi_1^{2j+1} + \cdots + \xi_n^{2j+1}).\]
\end{lemma}
\begin{proof}
See the Proposition 1 in \cite{CKSTT1}.
\end{proof}
We compute the time derivative of $E_I^2(t)$,
\[\frac{d}{dt}E_I^2(t) = \Lambda_2(m(\xi_1)m(\xi_2)\alpha_2) -i\Lambda_{3}([m(\xi_1)m(\xi_2+\xi_3)\set{\xi_2+\xi_3}]_{sym}).\]
The first term vanishes since $\xi_1 + \xi_2 = 0$ implies $\alpha_2 = 0$, and hence we have from the remainder that
\[\frac{d}{dt}E_I^2(t) = \Lambda_{3}(-i[m(\xi_1)m(\xi_2+\xi_3)\set{\xi_2+\xi_3}]_{sym}).\]
Let us denote
\[M_3(\xi_1,\xi_2,\xi_3) = -i[m(\xi_1)m(\xi_2+\xi_3)\set{\xi_2+\xi_3}]_{sym},\]
and define the new modified energy
\[E_I^3(t) = E_I^2(t) + \La_3(\sigma_3),\]
where the symmetric 3-multiplier $\sigma_3$ will achieve a cancellation. Using \eqref{eq:normal} again, we have
\[\frac{d}{dt}E_I^3(t) = \Lambda_3(M_3) + \Lambda_3(\sigma_3\alpha_3) + \Lambda_4(-i\frac32[\sigma_3(\xi_1,\xi_2,\xi_3+\xi_4)\set{\xi_3+\xi_4}]_{sym}).\]
Taking
\[\sigma_3 = -\frac{M_3}{\alpha_3}\]
gives a cancellation of the first two terms. With this choice, similarly as before, the time derivative of $E_I^3(t)$ is a 4-linear expression $\Lambda_4(M_4)$, where
\[M_4(\xi_1,\xi_2,\xi_3,\xi_4) = -i\frac32[\sigma_3(\xi_1,\xi_2,\xi_3+\xi_4)\set{\xi_3+\xi_4}]_{sym}.\]
In the same manner, we define the third modified energy by
\[E_I^4(t)= E_I^3(t) + \Lambda_4(\sigma_4)\]
with
\[\sigma_4 = -\frac{M_4}{\alpha_4},\]
and we obtain
\[\frac{d}{dt}E_I^4(t) = \La_5(M_5),\]
where
\[M_5(\xi_1,\cdots,\xi_5) = -2i[\sigma_4(\xi_1,\xi_2,\xi_3,\xi_4+\xi_5)\set{\xi_4+\xi_5}]_{sym}.\]
Under this setting, in order to show the global well-posedness for \eqref{eq:kdv}, we need to show that $E_I^2(t)$ is comparable to $E_I^4(t)$ at first, and next $E_I^4(t)$ is almost conserved. Let us start with obtaining some pointwise estimates which play a crucial role to show Proposition \ref{prop:comparable} and \ref{prop:quinti}, later. We define and state some calculus properties. If $m$ is of the form \eqref{eq:multiplier}, then $m^2$ satisfies 
\begin{equation}\label{eq:multiplier1}
\begin{split}
&m^2(\xi) \sim m^2(\xi') \quad \mbox{for} \quad |\xi| \sim |\xi'|,\\
&(m^2)'(\xi) = O\left(\frac{m^2(\xi)}{|\xi|}\right),\\
&(m^2)''(\xi) = O\left(\frac{m^2(\xi)}{|\xi|^2}\right),
\end{split}
\end{equation}
for all non-zero $\xi$. With this notion, we can observe two forms of the mean value formula which follow directly from the fundamental theorem of calculus. If $|\eta|,|\lambda| \ll |\xi|$, then,
\begin{equation}\label{eq:mean1}
m^2(\xi+\eta) - m^2(\xi) = O\left(|\eta|\frac{m^2(\xi)}{|\xi|}\right)
\tag{MVT}\end{equation}
and
\begin{equation}\label{eq:mean2}
m^2(\xi + \eta + \la) - m^2(\xi + \eta) - m^2(\xi + \la) + m^2(\xi) = O\left(|\eta||\la|\frac{m^2(\xi)}{|\xi|^2}\right).
\tag{DMVT}\end{equation}
From the following two lemmas, the multiplier $\sigma_3$ can be smoothly extended on $\R^3$ as in \cite{KT}.
\begin{lemma}\label{lem:extension}
Let $m$ is of the form \eqref{eq:multiplier}. Then for each dyadic $\lambda \le \eta$, there is an extension of $\sigma_3$ from the diagonal set
\[\set{(\xi_1,\xi_2,\xi_3) \in \Gamma_3 : |\xi_1|,|\xi_2| \sim \eta, \quad |\xi_3| \sim \lambda}\]
to the full dyadic set
\[\set{(\xi_1,\xi_2,\xi_3) \in \R^3 : |\xi_1|,|\xi_2| \sim \eta, \quad |\xi_3| \sim \lambda}\]
which satisfies the size and regularity conditions
\begin{equation}\label{eq:extension}
|\partial_{\xi_1}^{\beta_1}\partial_{\xi_2}^{\beta_2}\partial_{\xi_3}^{\beta_3}\sigma_3(\xi_1,\xi_2,\xi_3)| \lesssim m^2(\lambda)\eta^{-2j-\beta_1-\beta_2}\lambda^{-\beta_3}.
\end{equation}
The implicit constant does not depend on $\lambda, \eta$, but may depend on $\beta_i's$, $i=1,2,3$.
\end{lemma} 
\begin{proof}
We may assume that $|\xi_1| \gtrsim N$, otherwise $\sigma_3 \equiv 0$. Since $\xi_1 + \xi_2 + \xi_3 = 0$, we have from \eqref{eq:algebra1} that $\alpha_3 = i\xi_1\xi_2\xi_3Q_3(\xi_1,\xi_2,\xi_3)$ with the size $|\alpha_3| \sim \lambda \eta^{2j}$ and 
\[M_3(\xi_1,\xi_2,\xi_3) = -i[m(\xi_1)m(\xi_2+\xi_3)\set{\xi_2+\xi_3}]_{sym} = i(m^2(\xi_1)\xi_1+m^2(\xi_2)\xi_2+m^2(\xi_3)\xi_3).\]
If $\lambda \sim \eta$, we extend $\sigma_3$ by
\begin{equation}\label{eq:M3-1}
\sigma_3(\xi_1,\xi_2,\xi_3) = C \frac{m^2(\xi_1)\xi_1+m^2(\xi_2)\xi_2+m^2(\xi_3)\xi_3}{\xi_1\xi_2\xi_3Q_3(\xi_1,\xi_2,\xi_3)}
\end{equation}
and if $\lambda \ll \eta$, we extend $\sigma_3$ by
\begin{equation}\label{eq:M3-2}
\sigma_3(\xi_1,\xi_2,\xi_3) = C \frac{m^2(\xi_1)\xi_1-m^2(\xi_1 + \xi_3)\set{\xi_1 + \xi_3}+m^2(\xi_3)\xi_3}{\xi_1\xi_2\xi_3Q_3(\xi_1,\xi_2,\xi_3)}.
\end{equation}
From \eqref{eq:multiplier1} and \eqref{eq:mean1}, we have the desired result.
\end{proof}
From Lemma \ref{lem:extension}, we can easily obtain the pointwise bound for $M_3$. If $|\xi_1|\sim|\xi_2|\sim|\xi_3|$, we have directly 
\begin{equation}\label{eq:M3}
|M_3(\xi_1,\xi_2,\xi_3)| \lesssim m^2(\xi_3)|\xi_3|,
\end{equation}
from \eqref{eq:M3-1}, the triangle inequality and \eqref{eq:multiplier1}. Otherwise (i.e., if $|\xi_1| \sim |\xi_2| \gg |\xi_3|$), from \eqref{eq:M3-2}, we also have \eqref{eq:M3} by using \eqref{eq:mean1}.

Next, we give the pointwise estimate for $M_4$ which is the most important thing to show the almost conservation of $E_I^4(t)$.
\begin{lemma}\label{lem:M4}
Let $m$ is of the form \eqref{eq:multiplier}. For $N_i,N_{jk}$ dyadic and $N_1 \ge N_2 \ge N_3 \ge N_4$ where $|\xi_i| \sim N_i$ and $|\xi_j + \xi_k| \sim |N_{jk}|$, we have
\begin{equation}\label{eq:M4}
|M_4(\xi_1,\xi_2,\xi_3,\xi_4)| \lesssim \frac{|\alpha_4|m^2(\min(N_i, N_{jk}))}{(N+N_1)^j(N+N_2)^j(N+N_3)^{2j-1}(N+N_4)}.
\end{equation}
\end{lemma}
\begin{proof}
This proof is almost same as the proof of Proposition 5.3 in \cite{GW}. Moreover, when $N_4 \ll N/2$ and $N/2 \lesssim N_{12} < N_1/4$, we obtain exact \eqref{eq:M4} from the bound of 
\[\sigma_3(-\xi_3,-\xi_4,\xi_3+\xi_4)(\xi_3+\xi_4).\]
Indeed, let $N_4 \ll N/2$ ($\Rightarrow N_{13} \sim N_1$) and $N/2 \lesssim N_{12} < N_1/4$. From \eqref{eq:algebra2} and \eqref{eq:extension}, we get 
\[|\alpha_4| \sim |\xi_3+\xi_4|N_1^{2j} \gtrsim N/2 \cdot N_1^{2j}\]
and
\[|\sigma_3(-\xi_3,-\xi_4,\xi_3+\xi_4)(\xi_3+\xi_4)| \lesssim |\xi_3+\xi_4|N_3^{-2j} \lesssim N_3^{-2j+1},\]
respectively, and hence
\[\left|\frac{\sigma_3(-\xi_3,-\xi_4,\xi_3+\xi_4)(\xi_3+\xi_4)}{\alpha_4}\right| \lesssim \frac{1}{N_1^{2j}N_3^{2j-1}N}.\]
In the other cases, using \eqref{eq:algebra2}, \eqref{eq:mean1}, \eqref{eq:mean2} and \eqref{eq:extension}, one can obtain better bounds than the right-hand side of \eqref{eq:M4}. See \cite{GW} for the detailed proof.
\end{proof}
From the definition of $M_5$ and Lemma \ref{lem:M4}, we have the following pointwise bound for $M_5$:
\begin{lemma}\label{lem:M5}
Let $m$ is of the form \eqref{eq:multiplier}, and for $N_i,N_{jk}$ dyadic such that $|\xi_i|\sim N_i$ and $|\xi_j + \xi_l| = |\xi_{jl}| \sim N_{jl}$. Suppose $N_1 \ge N_2 \ge N_3 \ge N_4 \ge N_5$.

{\rm{(a)}} If $N_{12} \sim N_3$, then
\begin{equation}\label{eq:M5a}
|M_5(\xi_1,\xi_2,\xi_3,\xi_4,\xi_5)| \lesssim \frac{N_{12}}{(N+N_3)^{2j}(N+N_4)^{2j-1}(N+N_5)}.
\end{equation}

{\rm{(b)}} If $N_3 \sim N_4$, then for $N_5 \ge N_{12}$, we have 
\begin{equation}\label{eq:M5b}
|M_5(\xi_1,\xi_2,\xi_3,\xi_4,\xi_5)| \lesssim \frac{N_{12}}{(N+N_3)^{j}(N+N_4)^{j}(N+N_{12})^j(N+N_5)^j},
\end{equation}
and otherwise, we have
\begin{equation}\label{eq:M5c}
|M_5(\xi_1,\xi_2,\xi_3,\xi_4,\xi_5)| \lesssim \frac{N_{12}}{(N+N_3)^{j}(N+N_4)^{j}(N+N_{12})^{j+\frac12}(N+N_5)^{j-\frac12}}.
\end{equation}
\end{lemma}
\begin{proof}
Under the condition, we may assume that $N_1 \sim N_2 \gtrsim N$, since $M_5$ vanishes when $N_1 \ll N$. From the definition of $M_5$ and \eqref{eq:M4}, we have
\[|M_5(\xi_1,\xi_2,\xi_3,\xi_4,\xi_5)| \lesssim |\sigma_4(\xi_3,\xi_4,\xi_5,\xi_1+\xi_2)(\xi_1+\xi_2)|.\]
Using \eqref{eq:M4}, (a) can be easily proven. For (b), from the fact that if $N \ge M \ge 1$,
\[\frac{1}{N^{2j-1}M} \lesssim \frac{1}{N^{2j-1-\alpha}M^{1+\alpha}} \qquad \mbox{for} \quad 0 \le \alpha \le 2j-1\]
holds, and \eqref{eq:M4}, we can also easily prove \eqref{eq:M5b} and \eqref{eq:M5c}.
\end{proof}
Now, going back to the main parts in this section, we first prove that $E_I^2(t)$ is comparable to $E_I^4(t)$.
\begin{proposition}\label{prop:comparable}
Let $-\frac{j}{2} \le s < 0$ and $N \gg 1$. Then,
\begin{equation*}
|E_I^4(t) - E_I^2(t)| \lesssim \norm{Iu(t)}_{L^2}^3 + \norm{Iu(t)}_{L^2}^4.
\end{equation*}
\end{proposition}
\begin{proof}
In view of $E_I^4(t)$, we know that $E_I^4(t) = E_I^3(t) + \Lambda_4(\sigma_4) = E_I^2(t) + \Lambda_3(\sigma_3) + \Lambda_4(\sigma_4)$, so it suffices to show
\begin{equation}\label{eq:comparablepf1}
|\La_3(\sigma_3)| \lesssim \norm{Iu(t)}_{L^2}^3,
\end{equation}
and
\begin{equation}\label{eq:comparablepf2}
|\La_4(\sigma_4)| \lesssim \norm{Iu(t)}_{L^2}^4.
\end{equation}
Let us now use $k_i$ and $k_{jl}$ as variables instead of $\xi_i$ and $\xi_{jl}$ to prevent confusion from the notations throughout the paper.

We first show \eqref{eq:comparablepf1} and may assume that the $\wh{u}$ is nonnegative. Let us define $v = Iu$. From \eqref{eq:M3-1}, we need to show that
\begin{equation}\label{eq:comparablepf1-1}
\left|\La_3\left(\frac{m^2(k_1)k_1+m^2(k_2)k_2+m^2(k_3)k_3}{k_1k_2k_3Q_3(k_1,k_2,k_3)m(k_1)m(k_2)m(k_3)}\right)\right| \lesssim \norm{v}_{L^2}^3.
\end{equation}
We make a Littlewood-Paley decomposition and without loss of generality, assume $N_1 \ge N_2 \ge N_3$ for $|k_i| \sim N_i$ (dyadic). If $N_1 \le \frac{N}{2}$, then $\La_3$ vanishes, so we also assume $N_1 \sim N_2 \gtrsim N$. We consider two cases separately: $N_3 \ll N$ and $N_3 \gtrsim N$.

\textbf{Case I.} $N_3 \ll N$. From \eqref{eq:M3}, \eqref{eq:algebra1} and \eqref{eq:multiplier}, \eqref{eq:comparablepf1-1} is reduced to 
\[\sum_{N_1 \sim N_2 \ge N_3}\left|\La_3\left(\frac{N^{2s}}{N_1^{2(s+j)}} : v_1,v_2,v_3 \right)\right| \lesssim \prod_{j=1}^3\norm{v_j}_{L^2}.\]
Since $s \ge -\frac{j}{2}$, we know that $N_1^{-2(s+j)} \le N_1^{-j}$, so \eqref{eq:comparablepf1-1} is reduced as
\[N^{2s}\sum_{N_1 \sim N_2 \ge N_3}N_1^{-j}\int v_1v_2v_3 \: dx\lesssim \norm{v_1}_{L^2}\norm{v_2}_{L^2}\norm{v_3}_{L^2}.\]
We use the H\"older and the Sobolev inequalities to show
\begin{align*}
N^{2s}\sum_{N_1 \sim N_2 \ge N_3}N_1^{-j}\int v_1v_2v_3 \: dx&\lesssim N^{2s}\sum_{N_1 \sim N_2 \ge N_3}N_1^{-j}\norm{v_1}_{L^4}\norm{v_2}_{L^4}\norm{v_3}_{L^2}\\
&\lesssim N^{2s}\sum_{N_1 \sim N_2 \ge N_3}N_1^{-j}N_1^{\frac12}\norm{v_1}_{L^2}\norm{v_2}_{L^2}\norm{v_3}_{L^2},
\end{align*}
which implies the right-hand side of \eqref{eq:comparablepf1-1}.

\textbf{Case II.} $N_3 \gtrsim N$. From \eqref{eq:M3}, \eqref{eq:algebra1} and \eqref{eq:multiplier}, \eqref{eq:comparablepf1-1} is reduced to 
\[\sum_{N_1 \sim N_2 \ge N_3}\left|\La_3\left(\frac{N^{s}N_3^s}{N_1^{2(s+j)}} : v_1,v_2,v_3 \right)\right| \lesssim \prod_{j=1}^3\norm{v_j}_{L^2}.\]
Similarly as before, we have from the H\"older and the Sobolev inequalities that
\begin{align*}
N^{s}\sum_{N_1 \sim N_2 \ge N_3\gtrsim N}N_1^{-j}N_3^s\int v_1v_2v_3 \: dx&\lesssim N^{2s}\sum_{N_1 \sim N_2 \ge N_3}N_1^{-j}N_3^s\norm{v_1}_{L^4}\norm{v_2}_{L^4}\norm{v_3}_{L^2}\\
&\lesssim N^{2s}\sum_{N_1 \sim N_2 \ge N_3}N_1^{-j}N_3^sN_1^{\frac12}\norm{v_1}_{L^2}\norm{v_2}_{L^2}\norm{v_3}_{L^2},
\end{align*}
which also implies the right-hand side of \eqref{eq:comparablepf1-1}.

We turn to prove \eqref{eq:comparablepf2}. We make again a Littlewood-Paley decomposition and without loss of generality, assume $N_1 \ge N_2 \ge N_3 \ge N_4$ for $|k_i| \sim N_i$ (dyadic). If $N_1 \le \frac{N}{2}$, then $\La_4$ vanishes, so we also assume $N_1 \sim N_2 \gtrsim N$. From \eqref{eq:M4} and \eqref{eq:multiplier}, we need to show, similarly as \eqref{eq:comparablepf1}, that
\begin{equation}\label{eq:comparablepf2-1}
\begin{split}
\sum_{N_1 \sim N_2 \ge N_3\ge N_4}\left|\La_4\left(\frac{1}{(N+N_1)^{2j}(N+N_3)^{2j-1}(N+N_4)\prod_{i=1}^{4}m(k_i)}:v_1,v_2,v_3,v_4\right)\right|\\
\lesssim \prod_{i=1}^4\norm{v_i}_{L^2}.
\end{split}
\end{equation}
From \eqref{eq:multiplier}, we know
\[\left|\frac{1}{(N+N_1)^{2j}(N+N_3)^{2j-1}(N+N_4)\prod_{i=1}^{4}m(k_i)}\right| \lesssim \frac{N^{4s}}{N_1^{2(s+j)}\bra{N_3}^{s+2j-1}\bra{N_4}^{1+s}}.\]
From the H\"older and Sobolev inequalities, we have
\begin{align*}
\int v_1v_2v_3v_4\:dx &\lesssim \norm{v_1}_{L^2}\norm{v_2}_{L^2}\norm{v_3}_{L^{\infty}}\norm{v_4}_{L^{\infty}}\\
&\lesssim (N_3N_4)^{\frac12}\norm{v_1}_{L^2}\norm{v_2}_{L^2}\norm{v_3}_{L^2}\norm{v_4}_{L^2}.
\end{align*}
Hence, we finally obtain that
\begin{align*}
\mbox{LHS of } \eqref{eq:comparablepf2-1} &\lesssim \sum_{N_1 \sim N_2 \ge N_3\ge N_4} \frac{N^{4s}(N_3N_4)^{\frac12}}{N_1^{2(s+j)}\bra{N_3}^{s+2j-1}\bra{N_4}^{1+s}}\prod_{i=1}^{4}\norm{v_i}_{L^2}\\
&\lesssim N^{4s}\sum_{N_1~N_2}N_1^{-j}\norm{v_1}_{L^2}\norm{v_2}_{L^2} \sum_{N_2 \ge N_3 \ge N_4}\bra{N_3}^{-\frac32(j-1)}\bra{N_4}^{\frac{j}{2}+1}\norm{v_3}_{L^2}\norm{v_4}_{L^2},
\end{align*}
which implies the right-hand side of \eqref{eq:comparablepf2-1}, and hence we complete the proof of lemma.
\end{proof}
Now, we prove that $E_I^4(t)$ is the almost conserved quantity for $t \in (0,1]$. In order to show this, since
\begin{equation}\label{eq:almost}
|E_I^4(t) - E_I^4(0)| \lesssim \left|\int_0^1 \Lambda_5(M_5) \: dt\right|,
\end{equation}
we shall control the quintilinear form. 
\begin{proposition}\label{prop:quinti}
Let $-\frac{j}{2} \le s < 0$ and $N \gg 1$. Then,
\begin{equation*}
\left| \int_{0}^{1} \Lambda_5(M_5)\:dt\right| \lesssim N^{5s}\norm{Iu}_{Y^0}^5.
\end{equation*} 
\end{proposition}
\begin{proof}
We may assume that $\wt{u}$ be nonnegative and let us define $v= Iu$. Then, it suffices to show
\begin{equation}\label{eq:quinti1}
\int_0^1 \Lambda_5\left(\frac{M_5(k_1,k_2,k_3,k_4,k_5)}{m(k_1)m(k_2)m(k_3)m(k_4)m(k_5)}\right) \: dt \lesssim N^{5s}\norm{v}_{Y^0}^5.
\end{equation}
We make a Littlewood-Paley decomposition $v_i = P_{N_i}v$ for dyadic numbers $N_i$, $i=1,2,3,4,5$. Let $|k_i| \sim N_i$ and $|k_j+k_l| = |k_{jl}| \sim N_{jl}$, and without loss of generality, we may assume $N_1 \ge N_2 \ge N_3 \ge N_4 \ge N_5$. From Lemma \ref{lem:M5}, we only consider two cases that $N_3 \sim N_{12} \gtrsim N$ and $N_3 \sim N_4 \gtrsim N$. For the $N_3 \sim N_{12} \gtrsim N$ case, from \eqref{eq:multiplier} and \eqref{eq:M5a}, we have
\begin{align*}
\mbox{LHS of } \eqref{eq:quinti1} &\lesssim N^{5s}\int_0^1 \Lambda_5(|k_{12}|\bra{k_1}^{-s}\bra{k_2}^{-s}\bra{k_3}^{-s-2j}\bra{k_4}^{-s-2j+1}\bra{k_5}^{-s-1})\: dt\\
&\lesssim N^{5s}\sum_{N_1\sim N_2 \ge N_3 \ge N_4 \ge N_5}N_1^{-s}N_2^{-s}N_3^{-s-2j}\bra{N_4}^{-s-2j+1}\bra{N_5}^{-s-1}\norm{\px(v_1v_2)v_3v_4v_5}_{L_{t,x}^1}.
\end{align*}
From the Cauchy-Schwarz inequality, we reduce \eqref{eq:quinti1} to the following two estimates:
\begin{equation}\label{eq:quinti2}
N_1^{-s}N_2^{-s}\norm{\px(v_1v_2)}_{X^{s,-\frac12}} \lesssim N_1^{-2s-j}N_3^{s+\frac12}\norm{v_1}_{Y^0}\norm{v_2}_{Y^0},
\end{equation}
and
\begin{equation}\label{eq:quinti3}
N_3^{-s-2j}\bra{N_4}^{-s-2j+1}\bra{N_5}^{-s-1}\norm{v_3v_4v_5}_{X^{-s,\frac12}} \lesssim N_3^{-2s-j}\bra{N_4}^{-s-2j+\frac32}\bra{N_5}^{-s-\frac12}\prod_{j=3}^5\norm{v_j}_{Y^0}.
\end{equation}
Since $N_1,N_2,N_3 \ge N \gg 1$ and $N_3 \sim N_{12}$, we have \eqref{eq:quinti2} and \eqref{eq:quinti3} directly from \eqref{eq:bilinear2} and \eqref{eq:trilinear1}, respectively. Hence, we obtain
\begin{align*}
\mbox{LHS of } \eqref{eq:quinti1} &\lesssim N^{5s}\sum_{N_1\sim N_2 \ge N_3 \ge N_4 \ge N_5}N_1^{-2s-j}N_3^{\frac12-j}\bra{N_4}^{-s-2j+\frac32}\bra{N_5}^{-s-\frac12}\prod_{i=1}^{5}\norm{v_i}_{Y^0}\\
&\lesssim N^{5s}\norm{v}_{Y^0}^3\sum_{N_1}N_1^{-2s-j}\norm{v_1}_{Y^0}\norm{v_2}_{Y^0},
\end{align*}
which shows \eqref{eq:quinti1} for $ -\frac{j}{2} \le s < 0$.

Now, we consider the $N_3\sim N_4\sim N$ case. We further divide this case into $N_{12} \ge N_5$ and $N_5 \ge N_{12}$ cases. In these cases, since we have the upper bound of $M_5$ as \eqref{eq:M5b} and \eqref{eq:M5c} in Lemma \ref{lem:M5}, \eqref{eq:quinti1} is reduced by the same manner as above that
\begin{equation*}
N_1^{-s}N_2^{-s}\norm{\px(v_1v_2)}_{X^{s,-\frac12}} \lesssim N_1^{-2s-j}\bra{N_{12}}^{s+\frac12}\norm{v_1}_{Y^0}\norm{v_2}_{Y^0},
\end{equation*}
and
\begin{equation*}
N_3^{-s-j}N_4^{-s-j}\bra{N_5}^{-s-j}\norm{\prod_{j=3}^{5}v_i}_{X^{-s-j,\frac12}} \lesssim N_3^{-s}N_4^{-s-j}\bra{N_5}^{-s-j+1}\prod_{j=3}^5\norm{v_j}_{Y^0}
\end{equation*}
for $N_{12} \ge N_5$, and 
\begin{equation*}
N_3^{-s-j}N_4^{-s-j}\bra{N_5}^{-s-j+\frac12}\norm{\prod_{j=3}^{5}v_i}_{X^{-s-j-\frac12,\frac12}} \lesssim N_3^{-s}N_4^{-s-j}\bra{N_5}^{-s-j+1}\prod_{j=3}^5\norm{v_j}_{Y^0}
\end{equation*}
for otherwise. These estimates can be obtained directly from \eqref{eq:bilinear2}, \eqref{eq:trilinear2} and \eqref{eq:trilinear3}, respectively. Hence we obtain
\begin{align*}
\mbox{LHS of } \eqref{eq:quinti1} &\lesssim N^{5s}\sum_{N_1\sim N_2 \ge N_3 \sim N_4 \ge N_5}N_1^{-2s-j}N_3^{-2s-j}\bra{N_5}^{-s-j+1}\bra{N_{12}}^{s+\frac12}\prod_{i=1}^{5}\norm{v_i}_{Y^0}\\
&\lesssim N^{5s}\norm{v}_{Y^0}\sum_{N_1}N_1^{-2s-j}\norm{v_1}_{Y^0}^2\sum_{N_3}N_3^{-2s-j}\norm{v_3}_{Y^0}^2,
\end{align*}
which shows \eqref{eq:quinti1} for $ -\frac{j}{2} \le s < 0$, and hence we complete the proof.
\end{proof}
Finally, we sketch the proof of the global well-posedness by using the same argument in \cite{CKSTT1}. From the scaling property, \eqref{eq:kdv} with initial data $u_0 \in H_0^s(\T)$ is invariant under the following scaling:
\begin{equation*}\label{eq:scaling}
u_{\mu}(t,x) = \mu^{-2j}u(\mu^{-2j-1}t,\mu^{-1}x), \quad u_{0,\mu}(x) =\mu^{-2j}u_0(\mu^{-1}x),
\end{equation*}
with
\[\norm{u_{0,\mu}}_{H_0^s(\T_{\mu})} = \mu^{-s-2j+\frac12}\norm{u_0}_{H_0^s(\T)}.\]
Proposition \ref{prop:comparable} and \eqref{eq:almost} with Proposition \ref{prop:quinti} give 
\begin{equation}\label{eq:global1}
\sup_{0 \le t \le N^{-5s}}\norm{Iu(t)}_{L^2} \lesssim \norm{Iu(0)}_{L^2}.
\end{equation}
Moreover, a direct calculation also gives 
\begin{equation}\label{eq:global2}
\norm{Iu_{\mu}(0,\cdot)}_{L^2} \lesssim \mu^{-s-2j+\frac12}N^{-s}\norm{u_0}_{H_0^s},
\end{equation}
Taking $\mu \ge 1$ satisfying
\[\mu^{-s-2j+\frac12}N^{-s} = \epsilon_0 \ll 1,\]
implies with \eqref{eq:global1} and \eqref{eq:global2} that
\begin{align*}
\sup_{0 \le t \le T}\norm{u(t)}_{H_0^s} &\le \mu^{s+2j-\frac12}\sup_{0 \le t \le \mu^{2j+1}T}\norm{Iu_{\mu}(t)}_{L^2}\\
&\lesssim \mu^{s+2j-\frac12}\norm{Iu_{\mu}(0)}_{L^2} \\
&\lesssim N^{-s}\norm{u_0}_{H_0^s},
\end{align*}
when $\mu^{2j+1}T \le N^{-5s}$. Furthermore, for our global-in-time solution of \eqref{eq:kdv}, we have the uniform time growth bound of $H_0^s$-norm,
\begin{equation}\label{eq:uniform bound}
\sup_{0 \le t \le T} \norm{u(t)}_{H_0^s} \lesssim T^{\frac{2s+4j-1}{10s+16j-7}}\norm{u_0}_{H_0^s},
\end{equation}
for $-j/2 \le s < 0$.

\begin{remark}
In fact, in order to use the scaling argument in the proof of the global well-posedness, we need to consider the $\mu$-periodic function, $\mu \ge 1$. However, all estimates obtained in Section \ref{sec:bi-,tri-} for the global well-posedness do not depend on the $\mu$-scale, even though we prove those estimates under the $\mu$-periodic setting. Hence, we can use the scaling argument without further work. See Appendix \ref{sec:lambda} for the details.
\end{remark}

\section{Nonsqueezing property when $j\ge2$}\label{sec:nonsqueezing}

In this section, we prove the nonsqueezing property of \eqref{eq:kdv} when $j\ge 2$. As mentioned in Section \ref{subsec:Nonsq prop}, we first state the nonsqueezing property of (\ref{eq:truncated equation}) as an application of Theorem \ref{thm:finite dimensional nonsq.}.
\begin{lemma}\label{lem:Nonsqueezing of trun. fow}
Let $N \geq 1$, $ 0 <r <R$, $u_* \in P_{\le N} \HS\left(\MB{T}\right)$, $0< \left|k_0\right| \leq N$, $ z \in \MB{C}$ and $T>0$. Let $S_H^N \left(t\right) :  P_{\le N} \HS \left(\MB{T}\right) \to  P_{\le N} \HS \left(\MB{T}\right)$ be the solution map to \eqref{eq:truncated equation}.
Then 
\begin{equation*}
S_H^N \left(T\right) \left(B_R^N \left( u_* \right)\right) \not \subseteq C_{k_0,r}^N \left(z\right).
\end{equation*}
\end{lemma}

Our task, in this section, is to prove the closeness between two flows, $S_H\left(t\right)$ and $S_H^N\left(t\right)$. Since there are two differences between two flows, initial data and solution map, we can show the closeness by proving the following propositions, respectively:
\begin{proposition}\label{prop:approx}
Let $T>0$, and $N \gg 1$. Let $u_0, \underline u_0 \in \HS$ be such that $P_{\le2N} u_0 = P_{\le 2N}\underline u_0$. Then we have
\[\sup_{|t| \le T} \norm{P_{\le N}(S_{H}(t)u_0 - S_{H}(t)\underline u_0)}_{\HS} \leq C\Big(T,\left\|u_0\right\|_{\HS}\left\|,\underline u_0\right\|_{\HS}\Big) N^{-\sigma}\]
for some $\sigma>0$.
\end{proposition}
\begin{proposition}\label{prop:approx2}
Let $T>0$ and $N\gg1$. Let $u_0 \in \HS$ have Fourier transform supported in the range $\left|k\right| \le N$. Then we have
\[\sup_{|t| \le T} \norm{P_{\le N^{1/2}}(S_{H}(t)u_0 - S_{H}^{N}(t)u_0)}_{\HS} \leq C\Big(T,\left\|u_0\right\|_{\HS}\Big)N^{-\sigma}.\]
for some $\sigma >0$.
\end{proposition}
\begin{remark}
Proposition \ref{prop:approx} tells that change in the initial data at frequencies $\ge 2N$ does not significantly affect the solution at frequencies $\le N$ in the $\HS$.
\end{remark}
Now, we first prove Proposition \ref{prop:approx} by using estimates in Section \ref{sec:bi-,tri-}. We use the same argument in \cite{CKSTT3}. From the local well-posedness theory and the uniform bounds \eqref{eq:uniform bound}, Proposition \ref{prop:approx} can be reduced to the following lemma:
\begin{lemma}\label{prop:approx1}
Let $N' \gg 1$ and $u_0,\underline u_0 \in H_0^{-\frac12}$ satisfying $P_{\le N'}u_0 = P_{\le N'}\underline u_0$. Then if $T'$ is sufficiently small depending on $\norm{u_0}_{H_0^{-\frac12}}$ and $\norm{\underline u_0}_{H_0^{-\frac12}}$, we have
\begin{equation*}
\sup_{|t| \le T'} \norm{P_{\le N' - (N')^{1/2}}(S_{H}(t)u_0 - S_{H}(t)\underline u_0)}_{\HS} \leq  C\Big(\norm{u_0}_{H_0^{-\frac12}}, \norm{\underline u_0}_{H_0^{-\frac12}}\Big)\left(N'\right)^{-\sigma},
\end{equation*}
for some $\sigma > 0$.
\end{lemma}
\begin{remark}
In \cite{CKSTT3}, Colliander et al. explain why proving this proposition is enough to prove Proposition \ref{prop:approx}. See the section 5 in \cite{CKSTT3}.
\end{remark}

\begin{remark}
The proof of  Lemma \ref{prop:approx1} is easier and simpler than the proof of the Proposition 5.1 in \cite{CKSTT3}. Since we can obtain the good frequency decay bound from the bilinear estimate \eqref{eq:bilinear3}, no more techniques such as the Miura transform in \cite{CKSTT3} is required for our analysis as mentioned in Section \ref{subsec:Nonsq prop}. Moreover, since the right-hand side of \eqref{eq:bilinear3} has  the coefficient depending only on $N_{max}$, it is sufficient to separate $u$ into low and high frequencies different from the argument in \cite{CKSTT3}.
\end{remark}

To simplify our argument, consider
\begin{equation}\label{eq:kdv1}
u_t + (-1)^{j+1} \px^{2j+1}u = F(u,u),
\end{equation}
where $j>1$ and $F(u,v) = \frac12\px(uv)$ with 
\[F(u,v)\wh{\:}(k) = -\frac12\sum_{\substack{k_1,k_2 \in \Z^{\ast}\\k_1+k_2=k}} ik \wh{u}(k_1)\wh{v}(k_2).\]
Note that in contrast with analysis in \cite{CKSTT3}, since we control the quadratic form, the resonant term $F_0$ as in \cite{CKSTT3} is not considered.

\begin{proof}[Proof of Lemma \ref{prop:approx1}]
From the local well-posedness theory of \eqref{eq:kdv}, we have the local estimates
\begin{equation}\label{eq:small_1}
\norm{u}_{Y^{-\frac12}} + \norm{\underline u}_{Y^{-\frac12}} \le C,
\end{equation}
by choosing the sufficiently small time $T'$ depending on the $H_0^{-\frac12}$-norms of $u_0$ and $\underline u_0$.

Let $M \in \left[N' - \left(N'\right)^{\frac{1}{2}}, N'\right]$ be an integer. We separate $u$ as
\[u = u_{lo} +u_{hi},\]
where 
\[u_{lo} := P_{\le M}u, \hspace{1em} u_{hi}:= (1-P_{\le M})u.\]
From \eqref{eq:small_1}, we have
\begin{equation}\label{eq:split}
\norm{u_{lo}}_{Y^{-\frac12}}, ~ \norm{u_{hi}}_{Y^{-\frac12}} \le C.  
\end{equation}
We also split $\underline u$ and obtain the similar result as \eqref{eq:split} for $\underline u$. Applying $P_{\le M}$ to \eqref{eq:kdv1}, $u_{lo}$ obeys the equation
\begin{equation}\label{eq:low freq eq}
(\pt + \px^3)u_{lo} = P_{\le M}F(u,u).
\end{equation}
In order to control the right-hand side of \eqref{eq:low freq eq}  except for $F(u_{lo},u_{lo})$, define the \emph{error terms} to be any quantity with $Z^{-\frac12}$-norm of $O(\left(N'\right)^{-\sigma})$. From \eqref{eq:bilinear3}, we can easily know that all terms except for $P_{\le M}F(u_{lo},u_{lo})$ are error terms. 
Indeed, if the nonlinear term contains $u_{hi}$, then from the bilinear estimate \eqref{eq:bilinear3}, we have $M^{1-j}$ decay bound. Thus, $u_{lo}$ obeys the equation
\begin{equation}\label{eq:error1}
(\pt + \px^3)u_{lo} = P_{\le M}F(u_{lo},u_{lo}) + \mbox{error term}.
\end{equation} 
By the same manner, the function $\underline u_{lo}$ obeys the equation
\begin{equation}\label{eq:error2}
(\pt + \px^3)\underline u_{lo} = P_{\le M}F(\underline u_{lo},\underline u_{lo}) + \mbox{error term}.
\end{equation}
 
Since $u_{lo} = \underline u_{lo}$, we have from the standard local well-posedness theory that
\[\norm{u_{lo} - \underline u_{lo}}_{Y^{-\frac12}} \lesssim \left(N'\right)^{-\sigma},\]
which implies Lemma \ref{prop:approx1} by $Y^{s} \subset C_{t}H^s$.
\end{proof}
For Proposition \ref{prop:approx2}, we use the similar argument as in the proof of Proposition \ref{prop:approx}. The point of the proof of Proposition \ref{prop:approx} is to show \eqref{eq:error1} and \eqref{eq:error2}, and since $P_{\le N}P_{\le 2N} = P_{\le N}$, it suffices to obtain 
\[(\pt + \px^3)u_{lo} = P_{\le M}F(u_{lo},u_{lo}) + \mbox{error term}\]
and
\[(\pt + \px^3)v_{lo} = P_{\le M}F(v_{lo},v_{lo}) + \mbox{error term}.\]
However, those can be easily obtained by the same argument as in the proof of Lemma \ref{prop:approx1}. We omit the detailed proof of Proposition \ref{prop:approx2}.

 As the final stage to show the nonsqueezing property, we combine Lemma \ref{lem:Nonsqueezing of trun. fow}, Proposition \ref{prop:approx} and Proposition \ref{prop:approx2}. First of all, we show the following proposition:

\begin{proposition}\label{prop:main approx}
Let $k_0 \in \MB{Z}^*$, $T>0$, $A>0$, and $ 0 < \varepsilon \ll 1$. Then there exists a frequency $N_0 = N_0\left(k_0, T, \varepsilon, A\right) \gg \left|k_0\right|$ such that
\begin{equation*}
\left|k_0\right|^{-1/2} \left|\left(\SK\left(T\right)u_0\right)\hat{}\left({k_0}\right) - \left(\SN \left(T\right) u_0\right) \hat{}\left(k_0\right)\right| \ll \varepsilon
\end{equation*}
for all $N \geq N_0$ and all $u_0 \in B_A^N \left(0\right)$.
\end{proposition}

\begin{proof}
Let $\underline u_{0,N} = P_{\leq N} u_0$. By the triangle inequality, we have
\begin{align*}
&\left|k_0\right|^{-1/2} \left|\left(\SK\left(T\right)u_0\right)\hat{}\left({k_0}\right) - \left(\SN \left(T\right) u_0\right) \hat{}\left(k_0\right)\right| \\
\le& \left|k_0\right|^{-1/2} \left|\left(\SK\left(T\right)u_0\right)\hat{}\left({k_0}\right) - \left(\SK \left(T\right) \underline u_{0,N}\right) \hat{}\left(k_0\right)\right|  \\
+&\left|k_0\right|^{-1/2}  \left|\left(\SK\left(T\right)\underline u_{0,N}\right)\hat{}\left({k_0}\right) - \left(\SN \left(T\right) \underline u_{0,N}\right) \hat{}\left(k_0\right)\right|
\end{align*}
for $\left|k_0\right| \ll N$. 

From Proposition \ref{prop:approx2} and \ref{prop:approx}, we have
\[\left|k_0\right|^{-1/2} \left|\left(\SK\left(T\right)u_0\right)\hat{}\left({k_0}\right) - \left(\SK \left(T\right) \underline u_{0,N}\right) \hat{}\left(k_0\right)\right| \lesssim N^{-\sigma}\]
and
\[\left|k_0\right|^{-1/2} \left|\left(\SK\left(T\right)\underline u_{0,N}\right)\hat{}\left({k_0}\right) - \left(\SN \left(T\right) \underline u_{0,N}\right) \hat{}\left(k_0\right)\right| \lesssim N^{-\sigma},\]
respectively, for $N> N_0\left(k_0,T,\varepsilon,A\right)$ and $\left|k_0\right| \le N^{1/2}$. Thus, we complete the proof.
\end{proof}

Finally, we prove Theorem \ref{thm:Nonsqueezing thm} by combining with Lemma \ref{lem:Nonsqueezing of trun. fow} and Proposition \ref{prop:main approx}.

\begin{proof}[Proof of Theorem \ref{thm:Nonsqueezing thm}]
Choose $0 < \varepsilon < \frac{R-r}{2}$ and the ball $B^{\infty}_{R}\left(u_*\right) \subset B^{\infty}_{A}\left(0\right)$. We also choose $N > N_0\left(T,\varepsilon,k_0,A\right)$ so large that
\begin{equation*}
\left\|u_*- P_{\leq N} u_*\right\|_{H^{-1/2}_0} \leq \varepsilon.
\end{equation*}
From Lemma \ref{lem:Nonsqueezing of trun. fow}, we can find initial data $u_0 \in P_{\leq N} H^{-\frac{1}{2}}_0\left(\MB{T}\right)$ satisfying
\begin{equation*}
\left\|u_0 - u_*\right\|_{H^{-1/2}_0} \leq R-\varepsilon
\end{equation*}
and
\begin{equation*}
\left|k_0\right|^{-\frac{1}{2}} \left|\left(S_{H}^N\left(T\right)u_0\right)^{\wedge{}}\left(k_0\right)-z\right| > r+\varepsilon.
\end{equation*}
Then by the triangle inequality, we have
\begin{equation*}
\left\|u_0 - u_*\right\|_{H^{-1/2}_0} \leq \left\|u_0 - P_{\leq N }u_*\right\|_{H^{-1/2}_0} + \left\|P_{\leq N }u_*- u_*\right\|_{H^{-1/2}_0} \leq R.
\end{equation*}
Moreover, by the triangle inequality and Proposition \ref{prop:main approx}, we have
\begin{equation*}
\begin{split}
&\left|k_0\right|^{-\frac{1}{2}} \left|z-\left(S_{H}\left(T\right)u_0\right)^{\wedge{}}\left(k_0\right)\right| \\
\geq& \left|k_0\right|^{-\frac{1}{2}} \left[ \left|z-\left(S_{H}^N\left(T\right)u_0\right)^{\wedge{}}\left(k_0\right)\right| - \left|\left(S_{H}^N\left(T\right)u_0\right)^{\wedge{}}\left(k_0\right)-\left(S_{H}\left(T\right)u_0\right)^{\wedge{}}\left(k_0\right)\right|\right] \\
 >& r+\varepsilon -\varepsilon = r,
\end{split}
\end{equation*}
and this completes the proof.
\end{proof}

\appendix

\section{}\label{sec:lambda}

In this section, we will prove some multilinear estimates under the $\mu$-periodic setting in order to use the scaling argument in the proof of the global well-posedness in Section \ref{sec:global}. We start with introducing some notations adapted to the $2\pi\mu$-periodic setting.

We put $\T_{\mu} = [0,2\pi\mu]$ and $\Z_{\mu} := \set{k/\mu : k \in \Z}$. For a function $f$ on $\T_{\mu}$, we define
\[\int_{\T_{\mu}} f(x) \: dx := \int_0^{2\pi\mu} f(x) \: dx.\]
For a function $f$ on $\Z_{\mu}$, we define normalized counting measure $dk$:
\begin{equation}\label{eq:counting measure}
\int_{\Z_{\mu}} f(k) \: dk := \frac{1}{\mu}\sum_{k\in\Z_{\mu}} f(k)
\end{equation}
and $\ell_k^2(\mu)$ norm:
\[\norm{f}_{\ell_k^2(\mu)}^2 := \int_{\Z_{\mu}}|f(k)|^2 \: dk.\]
We define the Fourier transform of $f$ with respect to the spatial variable by
\[\wh{f}(k) := \frac{1}{\sqrt{2\pi}}\int_0^{2\pi\mu} e^{-ixk}f(x)\: dx, \hspace{3em} k \in \Z_{\mu},\]
and we have the Fourier inversion formula
\[f(x) := \frac{1}{\sqrt{2\pi}}\int_{\Z_{\mu}} e^{ixk}\wh{f}(k)\: dk, \hspace{3em} x \in \T_{\mu}.\]
Of course, we can naturally define the space-time Fourier transform similarly.

Then the usual properties of the Fourier transform hold:
\begin{equation}\label{eq:Plancherel}
\norm{f}_{L_x^2(\T_{\mu})} = \norm{\wh{f}}_{\ell_k^2(\mu)},
\end{equation}
\[\int_0^{2\pi\mu}f(x)\overline{g}(x) \: dx = \int_{\Z_{\mu}}\wh{f}(k)\overline{\wh{g}}(k) \: dk,\]
\[\wh{fg}(k) = (\wh{f} \ast \wh{g})(k) = \int_{\Z_{\mu}}\wh{f}(k-k_1)\wh{g}(k_1) \: dk_1\]
and for $m \in \Z_+$,
\begin{equation}\label{eq:derivatives}
\px^mf(x) = \int_{\Z_{\mu}}e^{ixk}(ik)^{m}\wh{f}(k) \: dk.
\end{equation}
Together with \eqref{eq:Plancherel} and \eqref{eq:derivatives}, we can define the Sobolev space $H^s(\T_{\mu})$ with the norm
\begin{equation}\label{eq:Hs norm}
\norm{f}_{H^s(\T_{\mu})} = \norm{\bra{k}^s\wh{f}(k)}_{\ell_k^2(\mu)}.
\end{equation}

We denote $X_{\mu}^{s,b}$, $Y_{\mu}^s$ and $Z_{\mu}^s$ by the modified spaces of $X^{s,b}$, $Y^s$ and $Z^s$ adapted to $\mu$-periodic setting, respectively.

Under those settings, we consider the scaling property. Let
\[u_{\mu}(t,x) = \mu^{-2j}u(\mu^{-2j-1}t,\mu^{-1}x), \]
what $u$ satisfies \eqref{eq:kdv} on $[0,T]$ with initial data $u_0 \in H^s(\T)$ is equivalent to what $u_{\mu}$ satisfies the same equation on $[0,\mu^{2j+1}T]$ with initial data $u_{0,\mu} \in H^{s}(\T_{\mu})$. By using \eqref{eq:counting measure} and \eqref{eq:Hs norm}, we obtain
\begin{equation}\label{eq:small}
\norm{u_{0,\mu}}_{H^s(\T_{\mu})} = (1+\mu^{-s})\mu^{-2j+1/2}\norm{u_0}_{H^s}(\T).
\end{equation}
In fact, since we may assume the mean-zero property, we can replace \eqref{eq:small} by
\begin{equation*}
\norm{u_{0,\mu}}_{H_0^s(\T_{\mu})} = \mu^{-2j-s+1/2}\norm{u_0}_{H_0^s}(\T).
\end{equation*}

We first restate several lemmas in Section \ref{sec:bi-,tri-} by modifying those adapted to $\mu$-periodic setting.

\begin{lemma}
For any function $u \in \T_{\mu} \times \R$, we have the $L^4$-Strichartz estimate for \eqref{eq:kdv}.
\[\norm{u}_{L_{t,x}^4(\T_{{\mu}})} \lesssim \norm{u}_{X_{\mu}^{0,\frac{j+1}{2(2j+1)}}}.\]
In particular, we have $\norm{u}_{L_{t,x}^4(\T_{{\mu}})} \lesssim \norm{u}_{X_{\mu}^{0,\frac13}}$.
\end{lemma}

\begin{proof}
The proof is similar as in \cite{CKSTT1} associated to the KdV equation. See Section 7 in \cite{CKSTT1} for the details.
\end{proof}

\begin{lemma}[Hirayama \cite{Hirayama}]
Let $j \in \N$ and $\mu \ge 1$. For $s \ge -j/2$, there exists $0 < \epsilon < 2j +s -1/2$ such that the following bilinear estimate holds:
\begin{equation*}
\norm{\ft^{-1}[\bra{\tau - k^{2j+1}}^{-1}\wt{\px uv}]}_{Z_{\mu}^s} \lesssim \mu^{\epsilon}\norm{u}_{Y_{\mu}^s}\norm{v}_{Y_{\mu}^s},
\end{equation*}
where the implicit constant dose not depend on $\mu$.
\end{lemma}

\begin{proof}
See \cite{Hirayama} for the proof.
\end{proof}
\begin{lemma}\label{lem:lambda-bilinear}
Let $j \in \N$ and $s \ge -j/2$. Let $u_i = P_{N_i}u$ and $|k_i|\sim N_i\ge1/\mu$, $i=1,2,3$. Then we have 
\[\norm{P_{N_3}\px(u_1u_2)}_{X_{\mu}^{s,-\frac12}} \lesssim (N_1N_2)^{-\frac12}N_3^{s+\frac12}N_{max}^{1-j}\norm{u_1}_{X_{\mu}^{0,\frac12}}\norm{u_2}_{X_{\mu}^{0,\frac12}}.\]
\end{lemma}

\begin{proof}
The proof is exact same as the proof of Lemma \ref{lem:bilinear2}, since we do not use the lower bound of the magnitude of the resonant function. 
\end{proof}
\begin{lemma}
Let $j \in \N$ and $-j/2 \le s < 0$. Let $u_i = P_{N_i}u$ and $|k_i|\sim N_i \ge 1/\mu$, $i=1,2,3$. Suppose that $k=k_1+k_2+k_3$, $|k_1|\ge|k_2|\ge|k_3|$ and $P_4(k_1,k_2,k_3,-k) \neq 0$, where $P_4$ is defined as in Lemma \ref{lem:algebra}. 

{\rm{(a)}} If $|k|\sim|k_1|$, then
\[\norm{u_1u_2u_3}_{X_{\mu}^{-s,\frac12}} \lesssim N_1^{-s+j}N_2^{\frac12}N_3^{\frac12}\norm{u_1}_{Y_{\mu}^0}\norm{u_2}_{Y_{\mu}^0}\norm{u_3}_{Y_{\mu}^0}\]

{\rm{(b)}} If $|k_1|\sim|k_2|$ and $j \ge 2$, then
\[\norm{u_1u_2u_3}_{X_{\mu}^{-s-j,\frac12}} \lesssim N_1^jN_3\norm{u_1}_{Y_{\mu}^0}\norm{u_2}_{Y_{\mu}^0}\norm{u_3}_{Y_{\mu}^0}, \hspace{2em} \mbox{for} \hspace{1em} |k_3| \ge |k|,\]
or
\[\norm{u_1u_2u_3}_{X_{\mu}^{-s-j-\frac12,\frac12}} \lesssim N_1^jN_3^{\frac12}\norm{u_1}_{Y_{\mu}^0}\norm{u_2}_{Y_{\mu}^0}\norm{u_3}_{Y_{\mu}^0}, \hspace{2em} \mbox{for}  \hspace{1em} |k| \ge |k_3|.\]
\end{lemma}

\begin{proof}
The proof is also similar to the proof of Lemma \ref{lem:trilinear}, due to the same reason in the proof of Lemma \ref{lem:lambda-bilinear}. 
\end{proof}
We remark that the restriction of low frequency does not affect the proof of global well-posedness.


\end{document}